\documentclass[11pt]{amsart}
\usepackage{url,enumitem, amsopn}
\usepackage{xcolor}
\usepackage{cleveref}
\usepackage{tikz, tikz-layers}
\usepackage{wrapfig}
\usepackage{lipsum}
\usepackage{pgfplots}
\pgfdeclarelayer{nodelayer}
\pgfdeclarelayer{edgelayer}
\pgfsetlayers{nodelayer,main,edgelayer}

\newcommand{\reals}{\mbox{$\mathbb R$}}

\newcommand{\nats}{\mbox{$\mathbb N$}}
\newcommand{\ints}{\mbox{$\mathbb Z$}}
\newcommand{\power}{\mbox{$\mathbb P$}}

\newcommand{\comment}[1]{}

\usepackage{comment}

\usepackage{amsmath, amsfonts, amssymb, latexsym, stmaryrd}
\usepackage{graphicx}                   

\def\squarebox#1{\hbox to #1{\hfill\vbox to #1{\vfill}}}
\def\qed{\hspace*{\fill}
        \vbox{\hrule\hbox{\vrule\squarebox{.667em}\vrule}\hrule}\smallskip}

\theoremstyle{plain}
\newtheorem{lemma}{Lemma}[section]
\newtheorem{theorem}[lemma]{Theorem}
\newtheorem{corollary}[lemma]{Corollary}
\newtheorem{proposition}[lemma]{Proposition}

\theoremstyle{definition}
\newtheorem{claim}[lemma]{Claim}
\newtheorem{observation}[lemma]{Observation}
\newtheorem{definition}[lemma]{Definition}
\newtheorem{convention}[lemma]{Convention}
\newtheorem{question}[lemma]{Question}

\newtheorem{example}[lemma]{Example}

\newtheorem*{rmk*}{Remark}
\newtheorem*{rmks*}{Remarks}
\newtheorem*{conventions*}{Conventions}
\newtheorem*{convention*}{Convention}

\def\squareforqed{\hbox{\rlap{$\sqcap$}$\sqcup$}}
\def\qed{\ifmmode\squareforqed\else{\unskip\nobreak\hfil
\penalty50\hskip1em\null\nobreak\hfil\squareforqed
\parfillskip=0pt\finalhyphendemerits=0\endgraf}\fi}

\newlength{\tablength}
\newlength{\spacelength}
\newcommand{\tabstar}{\hspace*{\tablength}}
\newcommand{\spacestar}{\hspace*{\spacelength}}
\def\obeytabs{\catcode`\^^I=\active}
{\obeytabs\global\let^^I=\tabstar}
{\obeyspaces\global\let =\spacestar}
\newenvironment{display}{\begingroup\obeylines\obeyspaces\obeytabs}{\endgroup}
\newenvironment{prog}{\begin{display}\parskip0pt\sf}{\end{display}}

\pgfplotsset{compat=1.18}

\author{Geir Agnarsson}
\address{Department of Mathematical Sciences \\ 
George Mason University \\ Fairfax, VA  22030}
\email{geir@math.gmu.edu}

\author{Elie Alhajjar}
\address{Department of Mathematical Sciences \\
U.~S.~Military Academy at West Point \\ West Point, NY 10996}
\email{elie.alhajjar@westpoint.edu}

\author{Aleyah Dawkins}
\address{Department of Mathematical Sciences \\ 
George Mason University \\ Fairfax, VA  22030}
\email{adawkin@gmu.edu}

\title{On locally finite ordered rooted trees and their rooted subtrees}



\date{\today}

\begin{document}

\begin{abstract}
In this article we compare the known dynamical polynomial time algorithm 
for the game-over attack strategy, to that of the brute force approach; 
of checking all the ordered rooted subtrees of a given tree that represents 
a given computer network. Our approach is purely enumerative and combinatorial 
in nature. We first revisit known results about a doubly exponential sequence
and generalize them. 
We then consider both finite and locally finite ordered rooted trees (LFOR-trees), 
and the class of their finite ordered rooted subtrees of bounded height, describing 
completely the LFOR-trees with no leaves where the number of ordered rooted 
subtrees of height at most $h$ are bounded by a polynomial in $h$. We finally  
consider general LFOR-trees where each level can have leaves and determine 
conditions for the number of ordered rooted subtrees of height at most $h$ 
to be bounded by a polynomial in $h$. 
\end{abstract} 

\maketitle

\section{Introduction}
\label{sec:intro}

This paper was originally inspired by some heuristic computations where
the brute force approach, of checking all possible scenarios, seemed to
be doing equally as well as known polynomial time algorithms. Although
our results are purely combinatorial and enumerative
in nature, we will briefly elaborate and explain, in the next few paragraphs,
how the problems we address in this paper, the definitions and our
discussions, stem from theoretical models of computer networks security. 

Security in computer networks, in particular cyber-security systems, have
long utilized layered security, commonly referred to as \emph{defense-in-depth}
strategies, where valuable assets are hidden behind many different
layers of the computer networks~\cite{Scarfone}. For example, a host-based
intrusion detection and prevention system may layer security through tools
such as host-based firewalls, encryption and shims, whereas a network-based
intrusion detection and prevention system may provide defense-in-depth
through an application layer, a transport layer, an Internet protocol (IP)
layer and a hardware layer. Many organizations also require the use of strong
authentication for remote access to intrusion detection and prevention
system components, such as two-factor authentication, which provides
an additional layer of security. Mathematically, rooted trees can be used
in various ways to model both attacks on such computer networks and also
the networks themselves. We briefly discuss two common ways trees are
currently being used.

{\sc One model:} To break into such a system and obtain
a valuable asset requires several layers of security to be penetrated.
One useful way to model such threats to a system is with decision making
trees called \emph{attack trees};
where the root corresponds to an attacker's goal and a child of a vertex
must be satisfied or ``fulfilled'' before that vertex can be reached.
Schneier was one of the first to use attack trees to model the security
of systems in a methodical way~\cite{Schneier}. This use of attack
trees to analyze the threat against a system has since been formalized
by Mauw and Oostdijk in~\cite{Mauw-Oostdijk}. More information on
attack trees and attack graphs as applied in cyber-security can be found
in~\cite{MEL, LI, ICLWB, LDB} and in the more software-based approach
in~\cite{SPEC}.

{\sc Another model:}
Rooted trees can also be used to model the computer network systems themselves
directly in a graph theoretic way. Namely, the impetus for the work in this
paper are the models of layered aspect of a cyber-security system as
originally introduced in~\cite{A-complexity} and~\cite{acta-cyb-1}, and
further analyzed in~\cite{acta-cyb-2}, in terms of a doubly-weighted
tree $(T,c,p,B,G)$, where
$T$ is a tree $T$ rooted at a vertex $r$,
$c$ is a \emph{penetrating-cost edge function} on $T$,
$p$ is a \emph{target-acquisition vertex function} on $T$,
$B$ is the \emph{attacker’s budget}; a positive real value representing
the total resources of the attacker and
$G$ is the \emph{game-over threshold}; a positive real value representing
the total value of critical components from the computer network that
suffices for a potential attacker to defeat the network completely.
In these three mentioned papers, the problem of determining the existence of,
and computing, a rooted subtree $T'$ of $T$, also rooted at $r$, where
(a) the edges of $T'$ sum up to a real number that is at most $B$, the
attacker's budget, and
(b) the non-root vertices of $T'$ sum up to a real number that is at
least $G$, the game-over threshold.
This model captures the notion that there is a cost associated with
penetrating each additional level of a system and that attackers have
finite resources to utilize in a cyber-attack. This model also highlights
the importance of designing systems that have fewest possible 
\emph{game-over components}; single components, represented as rooted subtrees $T'$, 
that achieve the game-over threshold, so important that once an adversary has obtained them,
the entire system is compromised.

In general, finding a game-over component is an intractable NP-complete
problem (even when the attacker knows the edge function $p$ and 
the vertex function $c$) and it is therefore highly unlikely to be amenable 
to polynomial-time solutions. For example, the simple case of a computer 
network consisting of a work station connected to $n$ distinct computers is 
represented by a doubly-weighted star $T$ rooted at $r$ and with $n$ leaves. 
In this case determining the existence of a rooted subtree $T'$ of $T$ within 
the attacker's budget $B$ of total vertex weight at least $G$ is equivalent to 
the well-known 0/1-Knapsack Decision  Problem~\cite[p.~65]{Garey}. However, 
there are many special cases where finding a game-over component 
(i.e.~a game-over rooted subtree $T'$ of $T$) can be solved in polynomial time 
by using recursion and dynamic programming~\cite{acta-cyb-1}. 

What is worth noting, and what is the main reason for this very paper, is that
for small networks with only a handful of connected computers, say a network
of $n\leq 25$ computers, then the direct brute force approach for both
determining the existence of and for finding the optimal game-over component
when we know it exists, works quite well~\cite{Ray}
and has computational complexity comparable with known polynomial
time methods from~\cite{acta-cyb-1}. When using the brute force method
for computing the game-over component, one must consider all rooted subtrees
$T'$ of a given tree $T$ that we assume has $n\in\nats$ non-root vertices.
However, when enumerating all these rooted subtrees of $T$, in fact all ordered
rooted subtrees of $T$, if we assume $T$ to be ordered, then it is often
more convenient to consider trees $T$ of a certain class as subtrees of
a special infinite tree of infinite height
(see Section~\ref{sec:infinite-OR-trees}) and focus on finite subtrees
of bounded height $h\in\nats$ instead of considering trees $T$ with at
most $n\in\nats$ non-root vertices. Such enumerations of various ordered
rooted trees have been considered by many. Aho and Sloane considered ordered
rooted binary trees, giving an exact formula in terms of the height $h$ of $T$
as a doubly exponential expression in terms of $h$~\cite{Aho-Sloane}.
Recently, a doubly exponential asymptotic formula for the number of
leaf-induced subtrees in a complete $k$-ary tree of a given height
was given~\cite{DossouOlory-Boadi}. Doubly exponential expressions do not
occur too often in computation simply since they grow so rapidly.
In addition to~\cite{Aho-Sloane,DossouOlory-Boadi} doubly exponential
sequences are also discussed and analyzed
in~\cite[p.~97, 100, 101, 109, 147]{Graham-Knuth-Patashnik}.

The remainder of this article is organized as follows:

In Section~\ref{sec:or-subtrees} we compute an exact formula for the number of
ordered rooted subtrees of a complete $k$-ary ordered rooted tree of
a given height $h$. This formula, given in Corollary~\ref{cor:t-closed-form},
is a doubly exponential expression and has a base number $c(k)\in\reals$
which we then analyze further for the first time. The main novel result in this section
describes the asymptotic behaviour of $c(k)$ as $k$ tends to infinity in
Theorem~\ref{thm:c(k)}.

In Section~\ref{sec:infinite-OR-trees} we consider locally finite ordered rooted
trees (or {\em LFOR-trees}, see definition~\ref{def:inf-tree}) and the class of their 
finite ordered rooted subtrees of bounded height. We then describe completely 
the LFOR-i-trees (LFOR-trees with no leaves) where
the number of ordered rooted subtrees of height at most $h$ are bounded
by a polynomial in $h$ for $h$ large enough (Theorem~\ref{thm:poly-iff}). These turn out to
be exactly the LFOR-i-trees of bounded width when they are viewed as posets
in the natural way. We then determine exactly the threshold function in terms
of the height $h$ for which an easily expressible bound from
Observation~\ref{obs:easy-s-bounds} will yield a natural power function
for the number of ordered rooted subtrees as described in
Proposition~\ref{prp:power-h}. 
The class of these threshold functions $h\mapsto\lambda_h$ can be expressed 
in terms of the Lambert W~function and we show they are all equivalent 
in that the ratio between any two of them tends to $1$ when $h$ tends to infinity 
in Corollary~\ref{cor:Lambert}. Other results on these threshold functions are also
presented. Specific examples of LFOR-i-trees are then provided
to illuminate the presented results. Finally, we prove some theoretical results about
the number of ordered rooted subtrees of height at most $h$ and its asymptotic behaviour
in Theorem~\ref{thm:poly-iff}.

In Section~\ref{sec:LFOR-trees} we discuss general LFOR-trees;
where each level can have some leaves. We define {\em thin} and {\em fat} LFOR-trees and determine
the leading coefficient of the number of ordered rooted subtrees of
height at most $h$; a polynomial in $h$ if the LFOR-tree is thin and
has a finite width as described in Theorem~\ref{thm:LFOR-thin-poly} and prove a corollary of this.

Finally, in Section~\ref{sec:summary} we summarize the main results and put them
into more general perspective.

\section{The number of ordered rooted subtrees of a given height}
\label{sec:or-subtrees}

The set of integers will be denoted by $\ints$, the
set of positive natural numbers by $\nats$ and the set of
non-negative integers $\nats\cup\{0\}$ by ${\nats}_0$. For each
$n\in\nats$ we let $[n] := \{1,\ldots,n\}$. The set of real numbers 
will be denoted by ${\reals}$ and the set of positive real numbers
by ${\reals}_+$. 
We write $f(x)\sim g(x)$ if 
$\lim_{n\rightarrow\infty}f(n)/g(n)= 1$ and we will further 
say (by abuse of this notation) that 
``$f(n)\sim g(n)$ for infinitely many $n$" if there is a 
sequence $(n_i)_{i\geq 0}$ of distinct natural numbers with
$\lim_{i\rightarrow\infty}f(n_i)/g(n_i) = 1$. Similarly, 
we denote by $f(n) = \Theta(g(n))$ if there are positive 
real constants $A$ and $B$ such that 
$Ag(n) \leq f(n) \leq Bg(n)$ for all $n\geq C$ where $C$ is some
constant. We will likewise say (by abuse of this notation) that ``$f(n) = \Theta(g(n))$ for infinitely many $n$" if there is a 
sequence of natural numbers $(n_i)_{i\geq 0}$ with
$Ag(n_i) \leq f(n_i) \leq Bg(n_i)$ for all $i\in\nats$. Finally, we
denote by $f(n) = o(g(n))$ if $f(n)/g(n)$ tends to zero when $n$ tends to infinity.

All graphs in this article will be simple graphs. In particular, a tree
will be a simple graph. An ordered rooted tree will be referred to as an
{\em OR-tree}. A subtree of an OR-tree $T$ is always itself an OR-tree,
but its root is not necessarily the root of $T$. A subtree of an OR-tree $T$
that has the same root as $T$ will be referred to as an {\em OR-subtree} of
$T$. Note that an OR-subtree is necessarily rooted, so in particular, an
OR-subtree has at least one vertex, namely the root.

{\sc Conventions:} We will to a large degree adhere to the usual convention of denoting
(i) the {\em order} of a graph (its number of vertices) by $n$,
(ii) the {\em size} of a graph (its number of edges) by $m$,
(iii) the {\em height} of an OR-tree (the largest number of edges
in a simple path from a leaf to its root) by $h$
(iv) the {\em distance} between two vertices by the number of edges in
the shortest path connecting them, and 
(v) the {\em level} $\ell(u)$ of a vertex $u$ in an OR-tree 
by the distance from $u$ to the root.

The following conventions assume $T$ is an OR-tree.
(vi) Denote by $S(T)$ the set of all OR-subtrees of $T$ and
let $s(T) = |S(T)|$ be the number of OR-subtrees of $T$.
(vii) The set of vertices of a tree $T$ will be denoted by $V(T)$.
(viii) The set of vertices of $T$ on level $h\in\nats$ will be denoted by
$V_h(T)$.
(ix) For a vertex $u\in V(T)$, denote the set of children of $u$ in $T$
by $C_T(u)$, or $C(u)$ when there is no danger of ambiguity. 
(x) For a vertex $u\in V(T)$, denote the set of descendants of $u$ in $T$
by $D_T(u)$, or $D(u)$ when there is no danger of ambiguity, and the
set of all descendants of $u$ together with $u$ by $D_T[u]$, or $D[u]$,
so $D[u] = D(u)\cup\{u\}$.
(xi) Denote the unique subtree in $T$ rooted at a vertex $u\in V(T)$
and containing all the descendants $D(u)$ by $T(u)$. So
$T(u)$ 
is the subtree of $T$ induced by $u$ and $D(u)$.

Every OR-tree is well suited for a variety of recursions.
Suppose OR-tree $T$ has root $r$ with $k$ children
$C(r) = \{u_1,\ldots, u_k\}$. If $|V(T)| = n$ and
$|V(T(u_i))| = n_i$ for each $i$, then $n = 1 + \sum_{i=1}^k n_i$. Since each
subtree $T(u_i)$ of $T$ is rooted at $u_i$, then clearly $n_i\geq 1$ for each
$i$. With this in mind and the obvious partition of $S(T)$ according to
whether or not each $u_i$ belongs to an OR-subtree, we have the following.
\begin{lemma}
\label{lmm:OR-subtrees}
For an OR-tree $T$ with a root $r$, let $C(r) = \{u_1,\ldots, u_k\}$
be the set of the children of $r$ in $T$.
In this case we have
$s(T)=\prod_{i=1}^k(1 + s(T(u_i)))$.
\end{lemma} 

Recall that for a $k\in\nats$, a {\em $k$-ary tree} is an OR-tree
where each vertex has at most $k$ children. A $2$-ary tree is then the
usual binary tree. As with binary trees, a $k$-ary tree is {\em full} (or
{\em regular}) if
every vertex has either $k$ or zero children, and a $k$-ary tree of height
$h$ is {\em complete} if at every level $\ell\in \{0,1,\ldots,h-1\}$ it
has precisely $k^\ell$ vertices, the maximum possible number of vertices
on level $\ell < h$. A complete $k$-ary tree of height $h$ whose level
$\ell = h$ also has $k^h$ vertices is {\em perfect}. A perfect $k$-ary
tree has (i) all its leaves on level $\ell = h$ and (ii) has the maximal
number $(k^{h+1}-1)/(k-1)$ of vertices among all $k$-ary trees of height $h$.

For $k\in\nats$, let ${\tau}_k(h)$ denote the perfect $k$-ary
tree of height $h$ and let $t_k(h) = |S({\tau}_k(h))|$ denote the number
of OR-subtrees of ${\tau}_k(h)$. Figure~\ref{3-ary-tree} shows the perfect $3$-ary tree ${\tau}_3(2)$.

\begin{center}
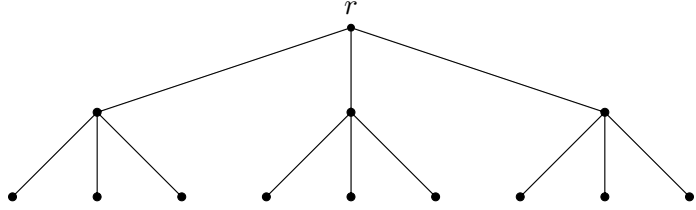
\begin{figure}

\begin{tikzpicture}[scale=.75]
 
\node [style={draw=black,circle,fill}, scale=.25, label=$r$] {}
     child {[fill] circle (2pt)
      child {[fill] circle (2pt)}
      child {[fill] circle (2pt)}
      child {[fill] circle (2pt)}
    }
    child[missing]
    child[missing]
    child {[fill] circle (2pt)
    child {[fill] circle (2pt)}
      child {[fill] circle (2pt)}
      child {[fill] circle (2pt)}
    }
    child[missing]
    child[missing]
    child {[fill] circle (2pt)
    child {[fill] circle (2pt)}
      child {[fill] circle (2pt)}
      child {[fill] circle (2pt)}
    };
 
\end{tikzpicture}
\caption{The $3$-ary tree $\tau_3(2)$ \label{3-ary-tree}}
\end{figure}
\end{center}

By Lemma~\ref{lmm:OR-subtrees}, the integer sequence $(t_k(h))_{h\geq 0}$
satisfies the recursion
\begin{equation}
\label{eqn:k-ary}
t_k(0) = 1, \ \ t_k(h+1) = (t_k(h) + 1)^k \mbox{ for $h\geq 0$.}
\end{equation}
These sequences $(t_k(h))_{h\geq 0}$, for varying $k\geq 2$, are
all clearly strictly increasing and 
are well-known. The first two of them are listed in
The On-Line Encyclopedia of Integer Sequences (OEIS)~\cite{oeis},
namely for $k = 2,3$, as the following rapidly growing sequences:
\[
\begin{array}{ll}
k = 2: & 1, 4, 25, 676, 458329, 210066388900, 44127887745906175987801,\ldots,\\
k = 3: & 1, 8, 729, 389017000, 58871587162270593034051001,\ldots,     
\end{array}
\]
listed as~\cite{A004019} and~\cite{A115590}, respectively.
Already for $k\geq 4$, the sequence $(t_k(h))_{h\geq 0}$ grows so fast
(known to be at least doubly exponential growth~\cite{Aho-Sloane})
that the sequences themselves are not listed in the OEIS.
For $k = 4,5$ the first few terms are in fact as follows:
\[
\begin{array}{ll}
k = 4: & 1,16,83521,48663522406470666256,\ldots, \\
k = 5: & 1,32,39135393,91801229324973413645775482048441660193,\ldots.
\end{array}
\]
Proceeding in the same way as was done in~\cite{Aho-Sloane} for the case
$k=2$, we will shortly derive a closed doubly exponential form for
the sequence $(t_k(h))_{h\geq 0}$ for general $k\in\nats$ as
\begin{equation}
\label{eqn:closed-form}
t_k(h) = \left\lfloor c(k)^{k^h}\right\rfloor - 1,
\end{equation}
where $c(k)\in{\reals}$ is a fixed positive real number
(see Corollary~\ref{cor:t-closed-form}), which can be
computed in an efficient manner. However, such an expression or formula
is more of theoretical value
since the evaluation of $t_k(h)$ for reasonably large values of $h$
will require very accurate approximation of the constants $c(k)$.
Indeed, as briefly discussed in~\cite[p.~109]{Graham-Knuth-Patashnik},
equations like (\ref{eqn:closed-form}) cannot really be considered to
be in closed form, because the constant $c(k)$ is computed from
the entries $t_k(h)$ themselves. 
However, expressions like (\ref{eqn:closed-form}) do shed an accurate light on
the asymptotic behavior of $t_k(h)$ for large $h$.

To the best of our knowledge, little is known about the real number
sequence $(c(k))_{k\geq 2}$ constituting the doubly exponential base number
for the sequence $(t_k(h))_{h\geq 0}$. The main part of this section will
be devoted to the derivation of some properties of $c(k)$
for varying $k\geq 2$. In order
to do that, however, we must first unravel the very method
from~\cite{Aho-Sloane} used to determine the existence of
the base number $c(2)$. Hence, we first proceed in a similar
fashion as in~\cite{Aho-Sloane} before continuing further
with more careful calculations to obtain more accuracy for
general $k\in\nats$.

Firstly, by definition $t_k(h) = |S(\tau_k(h))|$, for a fixed $h\geq 0$ 
the sequence $(t_k(h))_{k\geq 2}$ is an increasing sequence and strictly increasing 
for $h\geq 1$. Secondly, by induction on $h$ using (\ref{eqn:k-ary}) we then get 
the following.
\begin{claim}
\label{clm:increasing}
(i) For every fixed $k\geq 2$ the sequence $(t_k(h))_{h\geq 0}$ is strictly increasing.  
(ii) For every fixed $h\geq 1$ the sequence $(t_k(h))_{k\geq 2}$ is strictly increasing.
\end{claim}  
Secondly, we note that by trivially making the translated sequence
$({{t'}}_k(h))_{h\geq 0}$ from $(t_k(h))_{h\geq 0}$ by letting
${{t'}}_k(h) := {{t}}_k(h) + 1$ for each $h\geq 0$ we obtain from
(\ref{eqn:k-ary}) an equivalent recursion for ${{t'}}_h$:
\begin{equation}
\label{eqn:tprime-rec}
{{t'}}_k(0) = 2, \ \ {{t'}}_k(h+1) = {{t'}}_k(h)^k + 1 \mbox{ for $h\geq 1$.}
\end{equation}
For a given fixed $k\geq 2$, we have
\begin{equation}
\label{eqn:tprime-geq}
{{t'}}_k(h) \geq 2^{k^h} \mbox { for each $h\geq 0$.}
\end{equation}
The above inequality (\ref{eqn:tprime-geq}) is strict for $h\geq 1$ and is
obtained by a trivial induction on $h$: indeed ${{t'}}_k(0) = 2^{k^0}$
and for $h\geq 1$ we obtain
${{t'}}_k(h+1) = {{t'}}_k(h)^k + 1 > {{t'}}_k(h)^k \geq 2^{k^{h+1}}$.
Using the natural logarithm we let $x_k(h) = \log({{t'}}_k(h))$
and obtain from (\ref{eqn:tprime-geq}) that $x_k(h) \geq k^h\log2$
for each $h\geq 0$.

From (\ref{eqn:tprime-rec}) we obtain $x_k(0) = \log2$ and
$x_k(h+1) = \log({{t'}}_k(h)^k + 1) = kx_k(h) + {\alpha}_k(h)$,
where ${\alpha}_k(h) = \log(1 + 1/{{t'}}_k(h)^k)$ for each $h\geq 0$.
Using this recursion for $i = h, h-1, \ldots, 1$ and then summing up
we obtain
\begin{equation}
  \label{eqn:xequal-1}
x_k(h) = k^h\left(\log2 + \sum_{i=0}^{h-1}\frac{\alpha_k(i)}{k^{i+1}}\right).
\end{equation}
Since $\log(1+x)< x$ for positive real $x$ and $0 < {\alpha}_k(i) < 1/{{t'}}_k(i)^k \leq 1$, 
the series $\sum_{i=0}^{\infty}\frac{{\alpha}_k(i)}{k^{i+1}}$
converges absolutely to a positive real constant. 
Hence, we can make the following
definition.
\begin{definition}
\label{def:gamma(k)}
For each integer $k\geq 2$ let
\[
\gamma(k):=
\log2 + \sum_{i=0}^{\infty}\frac{{\alpha}_k(i)}{k^{i+1}} =
\log2 + \sum_{i=0}^{\infty}\frac{\log(1 + 1/{{t'}}_k(i)^k)}{k^{i+1}}
\in {\reals}.
\]
\end{definition}
In terms of $\gamma(k)$ from the above Definition~\ref{def:gamma(k)} we
get from (\ref{eqn:xequal-1}) the following equation
\begin{equation}
  \label{eqn:xequal-2}
x_k(h) = k^h\left(\gamma(k)-\sum_{i=h}^{\infty}\frac{\alpha_k(i)}{k^{i+1}}\right)
  = k^h\gamma(k) - \epsilon_k(h),
\end{equation}
where $\epsilon_k(h) := \sum_{i=0}^{\infty}\frac{{\alpha}_k(i+h)}{k^{i+1}}$.
We now briefly derive an upper bound for $\epsilon_k(h)$.
Again, since $0 < {\alpha}_k(i) < 1/{{t'}}_k(i)^k$ for each $i$, 
we obtain
\begin{equation}
  \label{eqn:e-sum-upperb}
\epsilon_k(h) < \sum_{i=0}^{\infty}\frac{1}{k^{i+1}{{t'}}_k(i+h)^k}.
\end{equation}
Since the sequence $({{t'}}_k(h))_{h\geq 0}$ is increasing we have
${{t'}}_k(i+h) \geq {{t'}}_k(h)$ for each $i\geq 0$ and by
(\ref{eqn:tprime-geq}) we therefore obtain the following concrete upper bound 
\begin{equation}
  \label{eqn:e-upperb}
  \epsilon_k(h) < \frac{1}{k{{t'}}_k(h)^k}\sum_{i=0}^{\infty}\frac{1}{k^i}
  = \frac{1}{{{t'}}_k(h)^k}\left(\frac{1}{k-1}\right)
  < \frac{1}{2^{k^{h+1}}}\left(\frac{1}{k-1}\right)
\end{equation}
for each $k\geq 2$ and $h\geq 0$. We can now make some observations.
Note that for $h=0$, as a consequence of (\ref{eqn:xequal-2}), that
$\gamma(k) = \log2 + \epsilon_k(0)$. This together with
(\ref{eqn:e-upperb}) yields $\gamma(k) < \log2 + \frac{1}{2^k(k-1)}$.
Since $\epsilon_k(h) > 0$ for each $h\geq 0$ we have
$\gamma(k) > \log2$ and hence, using the big-$O$ notation,
the following observation.
\begin{observation}
\label{obs:gamma-log2}  
The constant $\gamma(k)$ from Definition~\ref{def:gamma(k)} defined for each natural
number $k\geq 2$ satisfies $\gamma(k) = \log2 + O\left(\frac{1}{k2^k}\right)$,
and tends to $\log2$ from above when $k$ tends to infinity,
that is $\lim_{k\rightarrow\infty}\gamma(k) = {\log2}^{+}$.
\end{observation}
We now also can show that the base number $c(k)$ from our original sequence
$(t_k(h))_{h\geq 0}$ from (\ref{eqn:k-ary}) describing the number of
OR-subtrees of ${\tau}_k(h)$ is given by $c(k) = e^{\gamma(k)}$.
By Definition~\ref{def:gamma(k)} we can compute $c(k) = e^{\gamma(k)}$ 
for the first few values of $k$ with arbitrary precision, for example:
\begin{eqnarray*}
  c(2) & = & 2.2585184505894653988377962400637318724342746971851, \\
  c(3) & = & 2.0804006677503193521177452323719035237099784935372, \\
  c(4) & = & 2.0305447043459101719473131282027956101130315891183, \\
  c(5) & = & 2.0123466191423631126123265595251297217493484489300.
\end{eqnarray*}
By the definition of $x_k(h)$, (\ref{eqn:xequal-2}) and the fact that
$e^x > 1+x$ for any real $x$ we get
\begin{equation}
\label{eqn:delta-k}
  {{t'}}_k(h)
  = e^{x_k(h)}
  = (e^{\gamma_k(h)})^{k^h}e^{-\epsilon_k(h)}
  > (e^{\gamma_k(h)})^{k^h}(1  - \epsilon_k(h))
  = (e^{\gamma_k(h)})^{k^h} - \delta_k(h),
\end{equation}
where $\delta_k(h)$ satisfies the following by (\ref{eqn:e-upperb}), since $k\geq 2$,
\[
\delta_k(h) = (e^{\gamma_k(h)})^{k^h}\epsilon_k(h) 
<  \frac{(e^{\gamma_k(h)})^{k^h}}{2^{k^{h+1}}}\left(\frac{1}{k-1}\right)
\le  \left(\frac{e^{\gamma(k)}}{2^k}\right)^{k^h}.
\]
As noted just before Observation~\ref{obs:gamma-log2} we have since $k\geq 2$
that  $\gamma(k) < \log2 + 1/4$ and so $e^{\gamma(k)} < 2e^{1/4}$ which implies
that $\delta_k(h) < e^{1/4}/2 \approx 0.642 < 1$. Since $\epsilon_k(h)>0$ we
have by (\ref{eqn:delta-k}) that
$(e^{\gamma_k(h)})^{k^h} - \delta_k(h) < {t'}_k(h) < (e^{\gamma_k(h)})^{k^h}$
where $\delta_k(h) < 1$, and hence we have the formula
${t'}_k(h) = \lfloor (e^{\gamma_k(h)})^{k^h} \rfloor$.
By Observation~\ref{obs:gamma-log2} we then have the following, in particular
establishing (\ref{eqn:closed-form}).
\begin{corollary}
\label{cor:t-closed-form}
If $(t_k(h))_{h\geq 0}$ is the sequence describing the number of OR-subtrees of
the perfect $k$-ary tree ${\tau}_k(h)$, then
\[
t_k(h) = \left\lfloor c(k)^{k^h}\right\rfloor - 1,
\]
where $c(k) = e^{\gamma(k)}$. In particular $c(k)$ tends to $2$
from above when $k$ tends to infinity,
that is $\lim_{k\rightarrow\infty}c(k) = {2}^{+}$.
\end{corollary}
Slightly more can be said about the sequence $(c(k))_{k\geq 2}$ and how fast
it tends to $2$ when $k$ tends to infinity, which has some consequences
for our original sequence $(t_k(h))_{h\geq 0}$. We conclude this section
with some more accurate estimations of $\gamma(k)$ and hence $c(k)$ using
standard techniques and then state some remarks about OR-trees.

By Definition~\ref{def:gamma(k)} we firstly have the following lower bound 
\begin{equation}
\label{eqn:gamma-lowerb}  
\gamma(k)
= \log2 + \sum_{i=0}^{\infty}\frac{{\alpha}_k(i)}{k^{i+1}}
> \log2 + \frac{{\alpha}_k(0)}{k}
= \log2 + \frac{\log(1 + 1/2^k)}{k}
\end{equation}
and similarly we obtain
\begin{equation}
\label{eqn:gamma-equal}
\gamma(k)
= \log2 + \sum_{i=0}^{\infty}\frac{{\alpha}_k(i)}{k^{i+1}}
= \log2 + \frac{{\alpha}_0}{k}
  + \sum_{i=1}^{\infty}\frac{{\alpha}_k(i)}{k^{i+1}}
= \log2 + \frac{\log(1 + 1/2^k)}{k}
  + \sum_{i=1}^{\infty}\frac{{\alpha}_k(i)}{k^{i+1}}.
\end{equation}
Since $\log(1+x) < x$ for any positive real $x$, by
Claim~\ref{clm:increasing} and by (\ref{eqn:tprime-geq}) the above
sum can be bounded from above as follows
\begin{equation}
\label{eqn:sum-upperb}
\sum_{i=1}^{\infty}\frac{{\alpha}_k(i)}{k^{i+1}}
< \sum_{i=1}^{\infty}\frac{1/{t'}_k(i)^k}{k^{i+1}} 
< \sum_{i=1}^{\infty}\frac{1/2^{k^{i+1}}}{k^{i+1}} 
< \sum_{i=1}^{\infty}\frac{1/2^{k^{2}}}{k^{i+1}}
= \frac{1}{2^{k^{2}}k(k-1)}.
\end{equation}
From (\ref{eqn:gamma-lowerb}), (\ref{eqn:gamma-equal}) and
(\ref{eqn:sum-upperb}) we then have for $\gamma(k)$ that 
\begin{equation}
\label{eqn:gamma-tight}  
\log2 + \frac{\log(1 + 1/2^k)}{k} < \gamma(k) <
\log2 + \frac{\log(1 + 1/2^k)}{k} + \frac{1}{2^{k^{2}}k(k-1)},
\end{equation}
from which we obtain, by using the Taylor expansion of $e^x$
around $x = 0$, bounds for $c(k)$ that
\[
2(1 + 2^{-k})^{1/k} < c(k) <
2(1 + 2^{-k})^{1/k}\left(1 + O\left(\frac{1}{k^22^{k^2}}\right)\right).
\]
This, together with (\ref{eqn:gamma-tight}), yields
\begin{equation}
\label{eqn:gamma-c-big-O}
\gamma(k) = \log2 + \frac{\log(1 + 1/2^k)}{k}
+ O\left(\frac{1}{k^22^{k^{2}}}\right), \ \
c(k) = 2(1 + 2^{-k})^{1/k} + O\left(\frac{1}{k^22^{k^{2}}}\right).
\end{equation}
Lastly, by Claim~\ref{clm:increasing} the sequence
$({\alpha}_k(i))_{k\geq 2}$ is strictly decreasing and we obtain by definition
of $\gamma(k)$ and comparing term by term that if $k\leq k'$ then
\[
\gamma(k)
= \log2 + \sum_{i=0}^{\infty}\frac{{\alpha}_k(i)}{k^{i+1}} 
< \log2 + \sum_{i=0}^{\infty}\frac{{\alpha}_{k'}(i)}{{k'}^{i+1}} 
= \gamma(k').
\]
This together with (\ref{eqn:gamma-c-big-O}) yields the following summary.
\begin{theorem}
\label{thm:c(k)}
If $(t_k(h))_{h\geq 0}$ is the sequence describing the number of OR-subtrees of
the perfect $k$-ary tree ${\tau}_k(h)$, then the base number $c(k)$ from
Corollary~\ref{cor:t-closed-form} is strictly decreasing and satisfies
$c(k) = 2(1 + 2^{-k})^{1/k} + O(k^{-2}2^{-k^2})$ as $k\rightarrow\infty$.
\end{theorem}
{\sc Remark:} That the number of OR-subtrees of ${\tau}_k(h)$ is by
Corollary~\ref{cor:t-closed-form} given by
$t_k(h) = \lfloor c(k)^{k^h} \rfloor - 1$,
where $(c(k)_{k\geq 2})$ is strictly decreasing and tends to $2$ from above
by the above Theorem~\ref{thm:c(k)} when $k$ tends to infinity, might seem
counter intuitive at first sight. However, when viewing the function
$t_k(h)$ in terms of $n = |V({\tau}_k(h))|$, the number of vertices of
the perfect $k$-ary tree ${\tau}_k(h)$, we have
$n = n_k(h) = (k^{h+1}-1)/(k-1)$ and comparing the exponent of $t_k(h)$
to $n_k(h)$ we obtain
\begin{equation}
  \label{eqn:singly-exp}
\frac{\log(t_k(h))}{n_k(h)}\xrightarrow[h\to\infty]{}
\left(1 - \frac{1}{k}\right)\gamma(k) =
\left(1 - \frac{1}{k}\right)\log2 + O\left(\frac{1}{k2^k}\right),
\end{equation}
which, unlike $\gamma(k)$ and $c(k)$,
is a strictly {\em increasing} function in terms of $k$. Therefore, as a
singly exponential function in terms of the number of vertices
$n_k(h)$ of ${\tau}_k(h)$, the base number for $t_k(h)$ is
indeed an increasing function in terms of $n$, the number of vertices
of the tree ${\tau}_k(h)$.

Finally, we note that the left hand side of
(\ref{eqn:singly-exp}) tends upward to $\log2$ when $k$ tends to infinity,
the very ratio of $\log(|\power([n_k(h)])|)/n_k(h)$, where $\power([n_k(h)])$
is set of all subsets of $[n_k(h)] = \{1,2,\ldots,n_k(h)\}$. 

\section{Infinite ordered rooted trees}
\label{sec:infinite-OR-trees}

The perfect $k$-ary trees $\tau_k(h)$ of height $h$ considered in
the previous section can be viewed as the subtrees of an infinite
$k$-ary OR-tree where every vertex has exactly $k$ children. This infinite
tree has no leaves and is ``locally finite'' in the sense that each OR-subtree
of it induced by vertices of height $h$ or less is $\tau_k(h)$, a finite
tree. In this section we expand on this idea, considering classes of OR-trees that can be viewed
as OR-subtrees of a given infinite ``mother'' OR-tree $\mathbf{T}$.
We then investigate the OR-subtrees ${\mathbf{T}}_h$ of $\mathbf{T}$
that are induced by the finitely many vertices of $\mathbf{T}$ on levels
$h$ or less to describe when the number of OR-subtrees of height at most $h$ is a polynomial in $h$.
\begin{definition}
\label{def:inf-tree}
An {\em infinite rooted} tree $(\mathbf{T},r)$
is a simple connected graph $\mathbf{T}$ on infinitely many
vertices with a designated root $r\in V(\mathbf{T})$ and no cycles. 

An {\em infinite OR-tree} is an infinite rooted tree where the vertices on every
level have a designated left-to-right order.

An infinite rooted tree $\mathbf{T}$ is {\em locally finite} if for
each nonnegative integer $h$ the subtree ${\mathbf{T}}_h$ of $\mathbf{T}$
induced by all the vertices on level $h$ or less is a finite rooted tree.

A {\em LFOR-tree} is a locally finite infinite OR-tree.
\end{definition}

An example of an LFOR-tree is pictured in Figure~\ref{lfortree}.

\begin{figure}

\begin{center}

\begin{tikzpicture}[scale=.75]
 
\node [style={draw=black,circle,fill}, scale=.25, label=$r$] {}
     child {[fill] circle (2pt)
      child {[fill] circle (2pt)
      child {[fill] circle (2pt) 
      child {[fill] circle (2pt) 
      child {[fill] circle (2pt) {node[below]{$\vdots$}}}}}
      child {[fill] circle (2pt)}}
    }
    child[missing]
    child {[fill] circle (2pt)
    child {[fill] circle (2pt)
    child {[fill] circle (2pt)
    child {[fill] circle (2pt)child {[fill] circle (2pt) {node[below]{$\vdots$}}}}
    child {[fill] circle (2pt)}}}
    child {[fill] circle (2pt)
    child {[fill] circle (2pt)
    child {[fill] circle (2pt)
    child {[fill] circle (2pt) {node[below]{$\vdots$}}}}}}}
    child[missing]
    child[missing]
    child {[fill] circle (2pt)
    child {[fill] circle (2pt)
    child {[fill] circle (2pt)
    child {[fill] circle (2pt)
    child {[fill] circle (2pt) {node[below]{$\vdots$}}}}}}
    child[missing]
    child {[fill] circle (2pt)
    child {[fill] circle (2pt)
    child {[fill] circle (2pt)
    child {[fill] circle (2pt) {node[below]{$\vdots$}}}}
    child {[fill] circle (2pt)
    child {[fill] circle (2pt) {node[below]{$\vdots$}}}
    child {[fill] circle (2pt) {node[below]{$\vdots$}}}}
    child {[fill] circle (2pt)}}}}
    child[missing]
    child[missing]
    child {[fill] circle (2pt)
      child {[fill] circle (2pt)
      child {[fill] circle (2pt)}
      child {[fill] circle (2pt)}}
    };
 
\end{tikzpicture}

\caption{An LFOR-tree}\label{lfortree}

\end{center}

\end{figure}
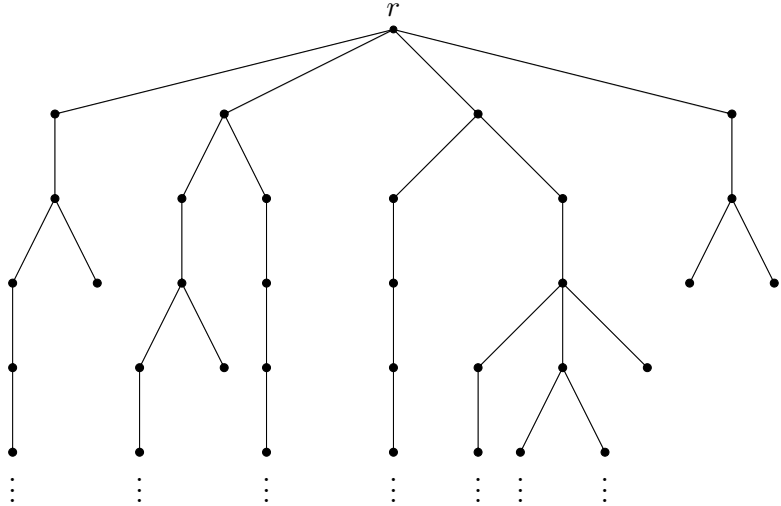

{\sc Remarks:}
(i) All the conventions and notations labeled (i) -- (xi)
in the beginning of the Section~\ref{sec:or-subtrees}
carry over to infinite rooted trees as well without any ambiguities.

(ii) Note that the condition of an infinite rooted tree $\mathbf{T}$
being locally finite is a stronger condition than $\mathbf{T}$ being
locally finite as a paritally ordered set, or {\em poset}, 
where $\mathbf{T}$ is viewed as a poset
in the natural way where the root is the sole maximum element,
since an infinite ordered tree with a vertex of infinite degree
certainly has the property that each interval
$[u,v] = \{w\in V(\mathbf{T}) : u\leq w\leq v\}$ is finite as each such interval contains the vertices
on the unique (and finite) path between $u$ and $v$ in $\mathbf{T}$.
\begin{example}
\label{exa:k-ary}
An infinite $k$-ary OR-tree where each vertex, including
the root $r$, has exactly $k$ children is an LFOR-tree with each
${\mathbf{T}}_h = \tau_k(h)$, the usual $k$-ary tree of height $h$.
\end{example}
\begin{example}
\label{exa:k-caterp}  
An infinite rooted tree where exactly one vertex on each
level, including the root on $0$-th level, has $k\geq 2$ children is
an LFOR-tree with each ${\mathbf{T}}_h$ being a caterpillar tree on
$hk+1$ vertices where $h(k-1)+1$ are leaves and the
central path has $h$ vertices.
\end{example}

As mentioned in the above remark, each rooted tree
(finite or not) is naturally a poset with its root as its maximum element. Recall an \emph{antichain} is a subset of a poset where any two distinct elements in the subset are incomparable. If the tree $T$ is finite, then the set of its leaves forms an
antichain when the tree is viewed as a poset. If $A\subseteq V(T)$ forms
an antichain and $u\in A$ is not a leaf of $T$, then $d_T(u)\geq 2$ and
there are at least $d_T(u) - 1$ leaves in $T$ that are descendants of $u$
in $T$. Denote their set by $L_T(u)$ and so $|L_T(u)| \geq d_T(u) - 1$.
In this way we obtain a new antichain $A' = (A\setminus \{u\})\cup L_T(u)$
of $T$ with (i) fewer non-leaves and (ii) of cardinality at least that of
$A$. Continuing in this manner,
exchanging each non-leaf in the antichain with its descendants leaves, we
obtain an antichain in $T$ consisting solely of leaves of cardinality at least
that of $A$.
\begin{convention}
\label{cvn:w-lambda}
(i) For a poset $(P,\leq)$ denote its {\em width};
the maximum cardinality of an antichain in $P$, by $w(P)$. (ii) For
a rooted tree $(T,r)$ denote by $\lambda(T)$ the number of leaves of $T$
(excluding the root.)
\end{convention}
We then have the following.
\begin{observation}
\label{obs:w=l}
If $(T,r)$ is a rooted tree equipped with the natural partial order where
the root $r$ is the sole maximum element, then
we have $w(T) = \lambda(T)$.
\end{observation}
By Dilworth's Theorem~\cite{Dilworth},~\cite{Trotter} (and~\cite{Perles}
for a particularly simply proof) any finite poset
$(P,\leq)$ can be decomposed into $w = w(P)$ chains of finite length
$P = C_1\cup\cdots\cup C_w$ and this is the minimum number of chains $P$ can
be decomposed into although the chains themselves are not unique.
For a rooted tree $(T,r)$ with its natural partial order, such a decomposition
is easy to obtain:
list the leaves $u_1,\ldots, u_w$ in any order.
Let $C_1$ consist of all the vertices in the unique path from $u_1$ to
the root $r$ in $T$. Having obtained the chains $C_1,\ldots,C_{i-1}$, we let
$C_i$ be the vertices in the longest path from $u_i$ toward the root $r$
that does not contain any vertices from $C_1\cup\cdots\cup C_{i-1}$.
We now consider such a decomposition for LFOR-trees.

Let $\mathbf{T}$ be an LFOR-tree. Every vertex of $\mathbf{T}$ either
has infinitely many descendants or finitely many descendants. 
\begin{convention}
\label{cvn:i-f}
Let $\mathbf{T}$ be an LFOR-tree.
A vertex of $\mathbf{T}$ is an {\em i-vertex} if it has infinitely 
many descendants, and it is an {\em f-vertex} if it has finitely 
many descendants. 
\end{convention}
Since $\mathbf{T}$ has infinitely many vertices, there are infinitely many
i-vertices of $\mathbf{T}$. Every i-vertex must have at least one
i-vertex child and every child of an f-vertex is also an f-vertex.
\begin{definition}
\label{def:i-tree}
An LFOR-tree where each of its vertices is an i-vertex is an {\em LFOR-i-tree}.
\end{definition}
An example of an LFOR-i-tree is pictured in Figure~\ref{lforitree}. LFOR-i-trees have some nice properties that we will now discuss. The
following is clear from the definition.
\begin{observation}
\label{obs:i-tree}
An LFOR-i-tree $\mathbf{T}$ has no leaves and hence each of its
finite subtrees ${\mathbf{T}}_h$ has all its leaves on its lowest 
level $h$.
\end{observation}

\begin{figure}

\begin{center}

\begin{tikzpicture}[scale=.75]
 
\node [style={draw=black,circle,fill}, scale=.25, label=$r$] {}
     child {[fill] circle (2pt)
      child {[fill] circle (2pt)
      child {[fill] circle (2pt) 
      child {[fill] circle (2pt) 
      child {[fill] circle (2pt) {node[below]{$\vdots$}}}}}}
    }
    child[missing]
    child {[fill] circle (2pt)
    child {[fill] circle (2pt)
    child {[fill] circle (2pt)
    child {[fill] circle (2pt)child {[fill] circle (2pt) {node[below]{$\vdots$}}}}}}
    child {[fill] circle (2pt)
    child {[fill] circle (2pt)
    child {[fill] circle (2pt)
    child {[fill] circle (2pt) {node[below]{$\vdots$}}}}}}}
    child[missing]
    child {[fill] circle (2pt)
    child {[fill] circle (2pt)
    child {[fill] circle (2pt)
    child {[fill] circle (2pt)
    child {[fill] circle (2pt) {node[below]{$\vdots$}}}}}}
    child {[fill] circle (2pt)
    child {[fill] circle (2pt)
    child {[fill] circle (2pt)
    child {[fill] circle (2pt) {node[below]{$\vdots$}}}}
    child {[fill] circle (2pt)
    child {[fill] circle (2pt) {node[below]{$\vdots$}}}
    child {[fill] circle (2pt) {node[below]{$\vdots$}}}}}}};
 
\end{tikzpicture}

\caption{An LFOR-i-tree}\label{lforitree}

\end{center}

\end{figure}
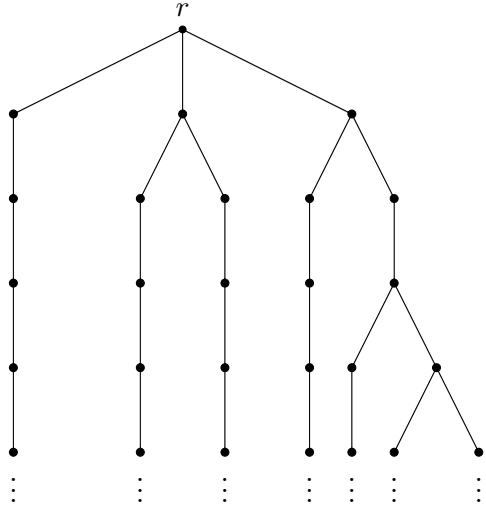

By the above Observation~\ref{obs:i-tree} the number of vertices of
an LFOR-i-tree $\mathbf{T}$ on level $h$ is then $\lambda(\mathbf{T}_h)$,
the number of leaves of $\mathbf{T}_h$. The integer sequence
$(\lambda(\mathbf{T}_h))_{h\geq 0}$ is increasing starting at $1$.
Let
\begin{equation}
\label{eqn:lambda}
\lambda(\mathbf{T}) = \lim_{h\rightarrow\infty}\lambda(\mathbf{T}_h),
\end{equation}  
so $\lambda(\mathbf{T})\in\nats\cup\{\infty\}$. By definition,
$\mathbf{T}$ has a countably infinite number of vertices and
since it is locally finite, each vertex has a finite degree.
Since $\mathbf{T}$ is an infinite OR-tree we have an ordering on its vertices
$V(\mathbf{T})$, lexicographically
(a) by level, and then 
(b) by their left-to-right order on each level.
Thus in this we have an ordering
$V(\mathbf{T}) = \{u_1,u_2,\ldots\}$. Similarly to
the case of a finite tree we can obtain a decomposition of $\mathbf{T}$
into infinite chains when $\mathbf{T}$ is equipped with its natural
partial order where each chain is order-reverse isomorphic to
$(\nats,\leq)$. The procedure is as follows:

Start with the root $u_1 = r$. Since $\mathbf{T}$ is an LFOR-i-tree there
 is an infinite chain $C_1$ emanating downward from $u_1$ in $\mathbf{T}$
that is formed by an infinite simple path from $u_1$ down through each level.
Having formed the infinite chains $C_1,\ldots,C_i$, we then proceed as follows:
(i) If $\mathbf{T}= C_1\cup\cdots\cup C_i$ then we are done.
Otherwise (ii) we choose the first vertex $u_{\nu(i+1)}\in V(\mathbf{T})$
that is not contained in $C_1\cup\cdots\cup C_i$. Since $\mathbf{T}$ is
an LFOR-i-tree there is an infinite chain $C_{i+1}$ emanating downward
from $u_{\nu(i+1)}$ in $\mathbf{T}$ and none of the vertices of this chain
is contained in any previous chain $C_1,\ldots,C_i$ (since $\mathbf{T}$
has no cycles). This process either terminates or goes on indefinitely.
Hence, the number of chains in this decomposition of $\mathbf{T}$ is
either finite or countably infinite. 

We note that
(i) each set of vertices on a fixed level of $\mathbf{T}$
must be contained in distinct chains, and so the number of chains
in the decomposition $\mathbf{T} = C_1\cup C_2\cup\cdots$ is at least
$\lambda(\mathbf{T})$ from (\ref{eqn:lambda}).
(ii) Conversely, for each chain $C_i$ there is a level
${\ell}(u_{\nu(i)})\in\nats$ of the maximum vertex of $C_i$,
such that $C_i$ contains exactly one vertex on each
level of ${\ell}(u_{\nu(i)})$ or greater. Note that the levels
$({\ell}(u_{\nu(i)})_{i\geq 1}$ of the maximum vertices $u_{\nu(i)}$ of the
chains $C_i$ form an increasing sequence in ${\nats}$. Therefore, for each
$i$, the chains $C_1,\ldots,C_i$ each contain exactly one vertex on level
${\ell}\geq {\ell}(u_{\nu(i)})$ and so $\lambda({\mathbf{T}}_{\ell})\geq i$.
As this holds for each $i$, we have the following analog of Dilworth's
Theorem for LFOR-i-trees. 
\begin{proposition}
\label{prp:cup-lambda}
Let $\mathbf{T}$ be an LFOR-i-tree equipped with the natural partial order that
makes the root the sole maximum element. If $\lambda(\mathbf{T})$ is as
in (\ref{eqn:lambda}), then $\mathbf{T}$ decomposes into infinite chains
\[
\mathbf{T} = \bigcup_{i=1}^{\lambda(\mathbf{T})} C_i,
\]
where each chain is order reverse isomorphic to $({\nats},\leq)$.
Further, this decomposition is minimal: there is no such decomposition
with fewer than $\lambda(\mathbf{T})$ chains. 
\end{proposition}
{\sc Remarks:}    
(i) We note that if $\lambda(\mathbf{T})\in\nats$ is finite, then the number
of chains in the decomposition in Proposition~\ref{prp:cup-lambda} is unique.
(ii) By a cardinality less than $\lambda(\mathbf{T})$ in
Proposition~\ref{prp:cup-lambda},
when $\lambda(\mathbf{T}) = \infty$ we mean a finite number from $\nats$.
(iii) Clearly, the chains $(C_i)_{i\geq 1}$ are not unique: a different
left-to-right order of vertices on each level will yield different chains.

\vspace{3 mm}

Consider the case when $\lambda(\mathbf{T})\in\nats$ is a finite number.
Since each antichain in $\mathbf{T}$ cannot contain more than one element
from each chain in the decomposition in Proposition~\ref{prp:cup-lambda}, we
have that the width of $\mathbf{T}$ satisfies
$w(\mathbf{T}) \leq \lambda(\mathbf{T})$. Also, by
Observation~\ref{obs:w=l} we have for each $h\in\nats$
that $w(\mathbf{T}) \geq w(\mathbf{T}_h) = \lambda(\mathbf{T}_h)$
for each $h$ and hence $w(\mathbf{T})\geq \lambda(\mathbf{T})$.
So, for an LFOR-i-tree $\mathbf{T}$ with $\lambda(\mathbf{T})$ finite,
we have $w(\mathbf{T}) = \lambda(\mathbf{T})$.

Consider an LFOR-i-tree $\mathbf{T}$ and the case when
$\lambda(\mathbf{T}) = \infty$ is infinite.
Since the set of each level $h$ forms an antichain in $\mathbf{T}$ we
clearly have that $\mathbf{T}$ contains an antichain of arbitrarily
large finite cardinality. We now argue that $\mathbf{T}$ contains an
infinite antichain.

First, we make a few observations. Since $\mathbf{T}$ is an LFOR-i-tree, we
then have $V_h(\mathbf{T}) = L(\mathbf{T}_h)$, the leaves of
$\mathbf{T}_h$.
Recall that $\mathbf{T}(u)$ denotes the LFOR-i-subtree of
$\mathbf{T}$ rooted at $u\in V(\mathbf{T})$ and contains all the
descendants of $u$ in $\mathbf{T}$.
For $h,k\in\nats$ with $h\leq k$
we have
$\lambda(\mathbf{T}_k) = \sum_{u\in V_h(\mathbf{T})}\lambda(\mathbf{T}(u)_k)$
and hence, by taking the limit $k\rightarrow\infty$ and let $h$ be arbitrary
but fixed, we get 
\begin{equation}
\label{eqn:l-partition}
\lambda(\mathbf{T}) = \sum_{u\in V_h(\mathbf{T})}\lambda(\mathbf{T}(u)).
\end{equation}
Since $\lambda(\mathbf{T}) = \infty$ and each $V_h(\mathbf{T})$ is a finite
set, then for sufficiently large $h_1$ we must have
(i) $\lambda(\mathbf{T}_{h_1}) = |V_{h_1}(\mathbf{T})| \geq 2$, and
(ii) $\lambda(\mathbf{T}(u_1)) = \infty$ for at least one
$u_1\in V_{h_1}(\mathbf{T})$ by (\ref{eqn:l-partition}).
Let $v_1\in V_{h_1}(\mathbf{T})\setminus\{u_1\}$.
Since $u_1$ and $v_1$ are on the same level $h$, the vertex $v_1$ is
incomparable to any vertex from $\mathbf{T}(u_1)$.
Assume we have chosen the integers $h_1,\ldots,h_i$ and
the vertices $u_1,\ldots,u_i$ and $v_1,\ldots,v_i$ in such a way that
$u_i$ and $v_i$ are distinct vertices from
$V_{h_i}(\mathbf{T}(u_{i-1}))$ and $\lambda(\mathbf{T}(u_i)) = \infty$.
Now let $h_{i+1}\in\nats$ be large enough for
$V_{h_{i+1}}(\mathbf{T}(u_i))$ to contain two or more vertices.
In this case we can choose two distinct vertices $u_{i+1}$ and $v_{i+1}$
from $V_{h_{i+1}}(\mathbf{T}(u_i))$ such that 
$\lambda(\mathbf{T}(u_{i+1})) = \infty$. We now have the integers
$h_1,\ldots,h_{i+1}$ and the vertices
$u_1,\ldots,u_{i+1}$ and $v_1,\ldots,v_{i+1}$. This process continues
indefinitely and we have the following.
\begin{claim}
\label{clm:inf-antichain}  
The infinite set $A = \{v_i : i\in\nats\}$ forms antichain
in $\mathbf{T}$ with it natural partial order.
\end{claim}
\begin{proof}
Let $i<j$ be natural numbers. Since $v_i$ and $u_i$ are distinct vertices
from $V_{h_i}(\mathbf{T}(u_{i-1}))$, then $v_i$ is incomparable to
any vertex from $\mathbf{T}(u_i)$ containing $u_i$ and all its
descendants in $\mathbf{T}$, including $v_j$. Hence, $v_i$ and $v_j$ are
incomparable and so $A$ is an infinite antichain in $\mathbf{T}$.
\end{proof}
By Claim~\ref{clm:inf-antichain} 
we then have for an LFOR-i-tree with $\lambda(\mathbf{T}) = \infty$
that the width of $\mathbf{T}$ is also infinity, and so
$w(\mathbf{T}) = \lambda(\mathbf{T})$. By
Proposition~\ref{prp:cup-lambda} we then have the following analog of
Dilworth's Theorem for LFOR-i-trees.
\begin{corollary}
\label{cor:cup-w}
Let $\mathbf{T}$ be an LFOR-i-tree equipped with the natural partial order that
makes the root the sole maximum element. If $\lambda(\mathbf{T})$ is as
in (\ref{eqn:lambda}) and $w(\mathbf{T})$ its width,
then $w(\mathbf{T}) = \lambda(\mathbf{T})$ and
$\mathbf{T}$ decomposes into infinite chains
\[
\mathbf{T} = \bigcup_{i=1}^{w(\mathbf{T})} C_i,
\]
where each chain is order reverse isomorphic to $({\nats},\leq)$ and
this is a characterization of LFOR-i-trees. 
Further, this decomposition is minimal: there is no such decomposition
with fewer than $w(\mathbf{T})$ chains. 
\end{corollary}

We now consider the function $s(\mathbf{T}_h)$, the number of OR-subtrees
of $\mathbf{T}_h$, for an LFOR-i-tree $\mathbf{T}$. If
$\lambda(\mathbf{T}) = 1$ then clearly $\mathbf{T}$ is an infinite rooted
path and $s(\mathbf{T}_h) = h+1$, a polynomial in $h$ of degree one and
leading coefficient of one.
Assume that $\lambda(\mathbf{T}) \ge 2$ is a finite natural number.
In this case, there is a smallest level in $\mathbf{T}$ that has at least
two vertices, say $\ell\geq 1$. Let $u$ be the singleton on level $\ell-1$
and let $u_1,\ldots,u_k$ be the children of $u$.
By Lemma~\ref{lmm:OR-subtrees} we get for any 
$h\geq\ell-1$ that
\begin{equation}
\label{eqn:s-poly}  
s(\mathbf{T}_h) = \ell - 1 + \prod_{i=1}^k(1 + s(\mathbf{T}(u_i)_{h-\ell})).
\end{equation}
Since each of the LFOR-trees $\mathbf{T}(u_i)$ is an LFOR-i-tree we have
by (\ref{eqn:l-partition}) that 
\[
\lambda(\mathbf{T}) = \sum_{i=1}^k\lambda(\mathbf{T}(u_i))
\]
where each $\lambda(\mathbf{T}(u_i))\geq 1$. By induction on
$\lambda(\mathbf{T})$ we can assume each $s(\mathbf{T}(u_i))$ to
be a polynomial in $h-\ell$, and hence in $h$ of degree
$\lambda(\mathbf{T}(u_i))$, and leading coefficient of one.
Since we have $s(\mathbf{T}_0) = 1$ for each LFOR-i-tree
$\mathbf{T}$, we have by (\ref{eqn:s-poly}) the following observation.
\begin{observation}
\label{obs:s-poly}
If $\mathbf{T}$ is an LFOR-i-tree with a finite $\lambda(\mathbf{T})\in\nats$,
then there is an $N\in\nats$ such that for all $h\geq N$
the number $s(\mathbf{T}_h)$ of OR-subtrees of height $h$ or less
is a monic polynomial in $h$ of degree $\lambda(\mathbf{T})$.
\end{observation}
{\sc Remark:} The number $N\in\nats$ in the above
Observation~\ref{obs:s-poly} can, in fact, be taken
to be the smallest $h$ such that
$\lambda(\mathbf{T}_h) = \lambda(\mathbf{T})$.
  
Assume now that $\lambda(\mathbf{T}) = \infty$. Let $d\in\nats$
and let $N\in\nats$ be such that $\lambda(\mathbf{T}_h) \geq d+1$
for all $h\geq N$. Restricting to all OR-subtrees of $\mathbf{T}_h$
that contain the subtree $\mathbf{T}_N$, we have for $h\geq N$
that $s(\mathbf{T}_h) \geq (h-N+1)^{d+1} \geq h^d$ for all sufficiently
large $h\geq N$. Since $d$ was arbitrary, this shows that
$s(\mathbf{T}_h)$ is larger than any polynomial when $h$ tends to infinity.
By this, together with Corollary~\ref{cor:cup-w} and
Observation~\ref{obs:s-poly}, we therefore have the following theorem.

\begin{theorem}
\label{thm:poly-iff}
Let $\mathbf{T}$ be an LFOR-i-tree and $\lambda(\mathbf{T})$ be as
in (\ref{eqn:lambda}). In this case the following are equivalent.
\begin{enumerate}
\item The width $w(\mathbf{T}) = \lambda(\mathbf{T})\in\nats$ is finite.
\item $\mathbf{T}$ can be decomposed into a finite number of infinite chains 
as in Observation~\ref{obs:s-poly} and this finite number is $w(\mathbf{T})$.
\item The function $s(\mathbf{T}_h)$ is bounded from above by a fixed
  polynomial in $h$ as $h\rightarrow\infty$.
\item The function $s(\mathbf{T}_h)$ is a monic polynomial in $h$ of
  degree $w(\mathbf{T})$ for all $h\geq N$ where $N\in\nats$ is a fixed number.
\end{enumerate}
\end{theorem}
Suppose an LFOR-i-tree $\mathbf{T}$ has a finite width $w = w(\mathbf{T})$ 
and let $N\in\nats$
be such that $\lambda(\mathbf{T}_h) = \lambda(\mathbf{T}) = w$ for all
$h\geq N$. Let $M = |V(\mathbf{T}_N)|$ be the number of vertices of
$\mathbf{T}_N$. In this case, $\mathbf{T}_h$ has exactly $n_h = M + w(h-N)$
vertices for each $h\geq N$. As the height $h$ can be written in terms of
the number $n_h$ of vertices of $\mathbf{T}_h$ we have
by the above Theorem~\ref{thm:poly-iff} the following.
\begin{corollary}
\label{cor:poly-i-vert}
An LFOR-i-tree $\mathbf{T}$ has a finite width as a poset if and only if
the number of OR-subtrees $s(\mathbf{T}_h)$ is a polynomial in terms of
$n_h = |V(\mathbf{T}_h)|$, the number of vertices of $\mathbf{T}_h$
for large enough $h$, in which case $s(\mathbf{T}_h)$ is a polynomial
of degree $w = w(\mathbf{T})$ with leading coefficient of $1/w^w$.
\end{corollary}

\subsection{Simple and natural bounds}
Suppose now an LFOR-i-tree $\mathbf{T}$ has an infinite width
$w(\mathbf{T}) = \lambda(\mathbf{T}) = \infty$. We first derive
some clear-cut upper and lower bounds for the number $s(\mathbf{T}_h)$
of OR-subtrees of $\mathbf{T}_h$ in terms of the height $h$
and then investigate when these bounds are asymptotically tight.

For each $h\in\nats$ let $\lambda_h = \lambda(\mathbf{T}_h)$.
In general, we clearly have $s(\mathbf{T}_h)\leq (h+1)^{\lambda_h}$ since
$(\lambda_h)_{h\geq 0}$ is an increasing sequence. Let $k\in\nats$ be fixed.
For each $h\geq k$ consider the collection of OR-subtrees of 
 $\mathbf{T}_h$ that (i) contain $\mathbf{T}_k$ as a subtree and
(ii) have exactly $\lambda_k$ leaves at their lowest level $h$.
This means that each OR-subtree in this collection is formed from
$\mathbf{T}_k$ by adding a path with at most $h-k$ vertices to each
of the $\lambda_k$ leaves of $\mathbf{T}_k$. Consequently, the
number of OR subtrees in this collection is exactly $(h-k+1)^{\lambda_k}$
and so $s(\mathbf{T}_h) \ge (h-k+1)^{\lambda_k}$ is We therefore have
the following simply expressible upper and lower bound. 
\begin{observation}
\label{obs:easy-s-bounds}
If $k\in\nats$ then for every $h\geq k$ we have
$(h-k+1)^{\lambda_k} \leq s(\mathbf{T}_h)\leq (h+1)^{\lambda_h}$.
\end{observation}
The bounds in the above Observation~\ref{obs:easy-s-bounds} are in a 
sense natural in that (i) they are obtained from considering disjoint 
paths emanating downward in an ``independent" way and hence yield the simplest
enumeration of subtrees, and (ii) they provide tight polynomial bounds
when $\lambda(\mathbf{T})$ is finite (see Example~\ref{exa:poly} here below.)

These bounds in Observation~\ref{obs:easy-s-bounds} may or may not be asymptotically tight.
This depends on the growth of the integer function $h\mapsto\lambda_h$.
However, their simplicity will be useful in expressing bounds that are
asymptotically tight for infinitely many values of $h$.
We note that Observation~\ref{obs:easy-s-bounds} will not always yield concrete
$\Theta$-bounds of $s(\mathbf{T}_h)$. However, 
Observation~\ref{obs:easy-s-bounds} can yield, as we will see,
an asymptotically tight upper
bound in the sense that the lower bound matches asymptotically
the upper bound for infinitely many values of $h$ as $h\rightarrow\infty$. 

\begin{example}
\label{exa:poly}
Let $\mathbf{T}$ be an LFOR-tree with $\lambda(\mathbf{T}) = \lambda\in\nats$ a constant
as in Theorem~\ref{thm:poly-iff}. Let $N\in\nats$ be such that $\lambda(\mathbf{T}_h) = \lambda$
for every $h\geq N$ and let $k\geq N$ be fixed. In this case the upper and lower 
bounds in Observation~\ref{obs:easy-s-bounds} for $k = N$ and $h\geq k$ are given by 
monic polynomials in $h$ of degree $\lambda$ as 
$(h-N+1)^{\lambda} \leq s(\mathbf{T}_h)\leq (h+1)^{\lambda}$ and therefore
\[
\lim_{h \rightarrow \infty} s(\mathbf{T}_{h})/h^{\lambda} = 1.
\]
Hence, the upper and lower bound agree asymptotically, that is $s(\mathbf{T}_h)\sim h^{\lambda}$
when $h$ tends to infinity.
\end{example}

\begin{example}
\label{exa:k-ary-not}
Consider the LFOR-i-tree $\mathbf{T}_h = \tau_k(h)$; the perfect
$k$-ary tree of height $h$. In this case we have by
Corollary~\ref{cor:t-closed-form} that
\[
s(\mathbf{T}_h) = s(\tau_k(h))
= t_k(h) = \left\lfloor c(k)^{k^h}\right\rfloor - 1
\approx c(k)^{k^h} \neq h^{k^h} =  h^{\lambda_h},
\]
for each $k\in\nats$.
Hence, $s(\mathbf{T}_h) \neq h^{\lambda_h}$ and is far from holding in this
case when $f(h) = k^h$ is an exponential function.
\end{example}

\begin{example}
\label{exa:lg-growth}
Suppose that for an LFOR-i-tree $\mathbf{T}$ we have
$\lambda(\mathbf{T}_h) = \lambda_h = \lfloor(\lg(h+2))\rfloor$ for $h\geq 0$.
For $k = 2^i - 2$ and each $h \in\{ 2^i - 2,\ldots,2^{i+1} - 3\}$ 
we have that $\lambda_k = \lambda_h = i$ and so by 
Observation~\ref{obs:easy-s-bounds} we get 
\[
\left(h - 2^i + 3\right)^i \leq s(\mathbf{T}_h) \leq (h+1)^i.
\]
Letting $h = 2^{i+1} - 3$ in the above display we have for each $i\in\nats$ that
\[
\left(2^i\right)^i \leq s(\mathbf{T}_h) \leq \left(2^{i+1} - 2\right)^i.
\]
Hence, we see that the ratio between the upper bound and the lower bound 
is $\sim 2^i$ as $i$ tends to infinity; so we cannot conclude any tight asymptotic 
bounds for $s(\mathbf{T}_h)$ from Observation~\ref{obs:easy-s-bounds}.
\end{example}

\begin{example}
\label{exa:lgeps-growth}
Similarly to the previous example, suppose that for an LFOR-i-tree $\mathbf{T}$ we have
$\lambda(\mathbf{T}_h) = \lambda_h = \lfloor(\lg(h+2))^\epsilon\rfloor$ where $0 < \epsilon < 1$.
In this case we have for $k = 2^{i^{1/\epsilon}} - 2$ and
each $h \in\{ 2^{i^{1/\epsilon}}-2,\ldots,2^{(i+1)^{1/\epsilon}} - 3\}$  we have
that $\lambda_k = \lambda_h = i$ and so by Observation~\ref{obs:easy-s-bounds} we get
\begin{equation}
\label{eqn:i-tight-e}    
\left(h - 2^{i^{1/\epsilon}} + 3\right)^i
\leq s(\mathbf{T}_h) \leq (h+1)^i.
\end{equation}
In particular, for $h = h_i := 2^{(i+1)^{1/\epsilon}} - 3$ we have for each $i$ that
\[
\left(2^{(i+1)^{1/\epsilon}} - 2^{i^{1/\epsilon}}\right)^i
\leq s(\mathbf{T}_{h_i}) \leq \left(2^{(i+1)^{1/\epsilon}} - 2\right)^i.
\]
Dividing through by $d_i := 2^{(i+1)^{1/\epsilon}}$ and noting that 
$(1+i)^{1/\epsilon} - i^{1/\epsilon} \sim i^{1/\epsilon -1}/\epsilon$ when $i$ tends to infinity,
we obtain, since $1/\epsilon - 1 > 0$, that 
\[
\lim_{i\rightarrow\infty}\frac{s(\mathbf{T}_{h_i})}{d_i}=1^{-}.
\]
This shows that for every $i\in\nats$ the upper bound in (\ref{eqn:i-tight-e}) is 
asymptotically tight for infinitely many values of $h$, namely 
$h = h_i = 2^{(i+1)^{1/\epsilon}} - 3$ for each $i$. Consequentially, we have 
for these infinitely many values of $h$ that $s(\mathbf{T}_h) \sim h^{\lambda_h}$,
just as in Example~\ref{exa:poly}.
\end{example}

The above example can be expanded on. Firstly, by mimicking the above
considerations and analyses of Example~\ref{exa:lgeps-growth} we can
obtain the same result from Observation~\ref{obs:easy-s-bounds} for
any increasing integer function $h\mapsto \lambda_h$ whose derivative 
grows slower than the derivative of $(\lg(h))^\epsilon$ where $0<\epsilon<1$.
Secondly, we can utilize the clear 
bounds from  Observation~\ref{obs:easy-s-bounds} to the fullest to
obtain the following for a general ``slow growing'' integer
function $h\mapsto \lambda_h$. 
\begin{convention}
\label{cvn:mock-inverse}
For a strictly increasing function $f : {\nats} \rightarrow {\nats}_0$ we let
its {\em mock-inverse} $f^{*} : {\nats}_0\rightarrow {\nats}$ denote the function 
defined in the following way: 

Let $f^{*}(h) = 1$ if $h \in \{0,\ldots,f(1)-1\}$\footnote{only needed if $f(1) > 0$} and for each $i\in{\nats}$ let 
\[
f^{*}(h) =  i \ \ \mbox{ whenever } h\in \{f(i),\ldots,f(i+1) - 1\}. 
\]
\end{convention}
{\sc Remark:} Note that 
(a) for each $i\in{\nats}$ we have $f^{*}(f(i)) = i$ and 
(b) for each $h\in{\nats}_0$ we have $f(f^*(h)) = h$ if and only if 
$h = f(i)$ for some $i\in\nats$. Hence the word ``mock" for mock-inverse. 
\begin{proposition}
\label{prp:power-h}  
Let $\mathbf{T}$ be an LFOR-i-tree with an infinite width
$w(\mathbf{T}) = \lambda(\mathbf{T}) = \infty$. 
Suppose $\lambda(\mathbf{T}_h)= \lambda_h = f^{*}(h)$
where $f : {\nats}\rightarrow {\nats}_0$ a strictly increasing function 
such that $f(i)/f(i+1) = o(1/i)$ when $i$ tends to infinity.
In this case the upper bound from Observation~\ref{obs:easy-s-bounds} is asymptotically
tight, meaning that $s(\mathbf{T}_h) \sim h^{\lambda_h}$ for infinitely
many $h$. 
\end{proposition}
\begin{proof}
Let $i\in{\nats}$. By definition of the mock-inverse $f^*$ of $f$ we have
for each $h\in \{f(i),\ldots,f(i+1)-1\}$ that $f^*(h) = i$ and hence by 
Observation~\ref{obs:easy-s-bounds} that for $h = f(i+1)-1$ we have
\[
(f(i+1)-f(i))^i \leq s(\mathbf{T}_{f(i+1)-1})\leq f(i+1)^i.
\]
We see that the ratio of the lower bound to 
the upper bound is given by
\begin{equation}
\label{eqn:fi}
\left(\frac{f(i+1) - f(i)}{f(i+1)}\right)^i = 
\left(1 - \frac{f(i)}{f(i+1)}\right)^i = 
\left(\left(1 - \frac{f(i)}{f(i+1)}\right)^{\frac{f(i+1)}{f(i)}}\right)^{\frac{if(i)}{f(i+1)}}.
\end{equation}
By assumption of $f$ we have $if(i)/f(i+1)$ tends to zero as $i$ does. In particular
both $f(i+1)/f(i)$ and hence $f(i)$ tend to infinity as $i$ does.
Noting that $(1 - 1/k)^k\rightarrow e^{-1}$ when $k$ tends
to infinity, the limit from (\ref{eqn:fi}), when $i$ tends to infinity, is $1$. This proves that $s(\mathbf{T}_h) \sim (h+1)^{f^*(h)}$ for all $h = f(i+1) - 1$ where $i\in\nats$. Finally, for $h = f(i+1)-1$ we have $(1 + 1/h)^i$ tends to $1$ as $i$ 
tends to infinity. This complete the proof.
\end{proof}
In light of the condition on the function $f$ in Proposition~\ref{prp:power-h}, we 
will make the following convention.
\begin{convention}
\label{cvn:well-behaved}
We will call a function $f : \nats\rightarrow\nats$ {\em well-behaved} if the ratio
$if(i)/f(i+1)$ tends to a limit  $\alpha \in \{0\}\cup {\reals}_+\cup\{\infty\}$
when $i$ tends to infinity. Further, we call $f$ a {\em threshold function} if 
$\alpha\in {\reals}_+$.
\end{convention}
Restricting to well-behaved functions eliminates only functions with strange zigzag shapes and
unusual growth properties. All smooth analytic functions with monotone derivatives yield
well-behaved functions when restricting to the integers. 
When restricting to well-behaved functions we note that Observation~\ref{obs:easy-s-bounds} 
yields tight bounds {\em if and only if} $\alpha = 0$. In other words, we obtain tight asymptotic bounds from Observation~\ref{obs:easy-s-bounds} if and only
if $\frac{f(i)}{(i-1)f(i-1)} = p(i-1)$ for some function $p : {\nats}\rightarrow {\nats}_0$
with $p(i)\rightarrow\infty$. Letting $p!(n) := p(n)p(n-1)\cdots p(1)$ we then have
\[
\frac{f(n)}{(n-1)!f(1)} 
= \prod_{i = 2}^{n}\frac{f(i)}{(i-1)f(i-1)} 
= p!(n-1)
\]
and since $f(1) > 0$ we can state the following.
\begin{observation}
\label{obs:f-iff}
Among well behaved functions $f$ we have that Observation~\ref{obs:easy-s-bounds} yields a 
tight asymptotic bound, so $s(\mathbf{T}_h) \sim h^{\lambda_h}$ for infinitely many $h$,
where $\lambda_h = f^{*}(h)$, if and only if 
$f(n) = (n-1)!p!(n-1)$ for some function $p:\nats\rightarrow {\nats}_0$ that tends to infinity 
when $n$ tends to infinity.
\end{observation}
On the other hand, when considering well-behaved functions $f$ such that
$\alpha = \infty$ we see from (\ref{eqn:fi}) in 
the above proof that the ratio of the lower bound to 
the upper bound from Observation~\ref{obs:easy-s-bounds} tends to zero and hence we cannot 
obtain any asymptotic bounds for $s(\mathbf{T}_h)$ for large $h$ from 
Observation~\ref{obs:easy-s-bounds}.

Consider now the threshold functions $f$ with $\alpha\in {\reals}_+$.
In this case we see from (\ref{eqn:fi}) in 
the above proof that the ratio of the lower bound to 
the upper bound from Observation~\ref{obs:easy-s-bounds} tends to $e^{-\alpha}\in {\reals}_+$.
Although we don't obtain explicit tight asymptotic bounds for $s(\mathbf{T}_h)$ we do 
obtain tight asymptotic bounds up to a fixed constant (namely $e^{-\alpha}$) for 
$s(\mathbf{T}_h)$ for infinitely many $h$. 
\begin{observation}
\label{obs:alpha}
Among threshold functions $f$ with 
$\frac{f(i)}{(i-1)f(i-1)}\rightarrow {\alpha}^{-1}:= \beta \in {\reals}_+$
we have that $s(\mathbf{T}_h) = \Theta(h^{\lambda_h})$ for all $h = f(i+1) - 1$ where $i\in\nats$
and $\lambda_h = f^*(h)$. 
\end{observation}
Recall that if $(a_n)_{n\geq 0}$ is a sequence of real numbers and 
$a_n\rightarrow\beta$, then the geometric mean 
$\sqrt[n]{a_1\cdots a_n}\rightarrow\beta$ as $n$ tends to infinity. In particular,
if $f$ is a threshold function with $\alpha\in {\reals}_+$, then we obtain
\[
\sqrt[n]{\frac{f(n)}{(n-1)!f(1)}} = 
\sqrt[n]{\prod_{i = 2}^{n}\frac{f(i)}{(i-1)f(i-1)}} \rightarrow\beta
\]
as $n$ tends to infinity, and hence for every $\epsilon > 0$ we then have 
\[
((\beta-\epsilon)f(1))^n(n-1)! < f(n) < ((\beta+\epsilon)f(1))^n(n-1)!
\]
for all sufficiently large $n\in\nats$. We therefore have a clear form of these threshold 
functions for which $\frac{f(i)}{(i-1)f(i-1)}$ have a positive real limit when $i$ tends to 
infinity.
\begin{claim}
\label{clm:threshold}
Every threshold function $f(n)$ with $\alpha\in {\reals}_+$ is always 
sandwiched between two functions $c_1^n(n-1)!$ and $c_2^n(n-1)!$ for all 
sufficiently large $n$ where $c_1 < c_2$ are positive real constants.
Further, these constants can be chosen so that $c_2 - c_1$ is arbitrarily small. 
\end{claim}
By the above Claim~\ref{clm:threshold} we see that threshold functions of the 
form $f(n) = c^n(n-1)!$
are of particular interest, as they collectively, for varying $c\in {\reals}_+$, 
form a class of boundary functions among all threshold functions.

Recall the Stirling's Approximation Formula 
$n! \sim \sqrt{2\pi n}\left(\frac{n}{e}\right)^n$. In fact, for all $n\in\nats$ we 
have $n! > \sqrt{2\pi n}\left(\frac{n}{e}\right)^n$, and so every threshold function
$f$ of the form $f_c(n) = c^n(n-1)!$ satisfies 
\[
f_c(n) \sim c^n\sqrt{2\pi(n-1)}\left(\frac{n-1}{e}\right)^{n-1} 
= c\sqrt{2\pi}(n-1)^{n-1/2}\left(\frac{c}{e}\right)^{n-1}.
\]
Note also that $(n - 1)^{n-1/2}\sim e^{-1/2}(n - 1/2)^{n - 1/2}$ as $n\rightarrow\infty$, 
therefore this together with the above display yields that
\begin{equation}
\label{eqn:f-sim}
f_c(n) \sim c\sqrt{2\pi}(n-1/2)^{n-1/2}\left(\frac{c}{e}\right)^{n-1/2}
\left(\frac{e}{c}\right)^{1/2}.
\end{equation}
We will use this from to express the mock-inverse $f^*$ explicitly and thereby
obtain a formula for $f^*$ for these threshold functions $f$. But first we 
present a result in elementary function theory. Since we have not seen it
explicitly stated anywhere we state it as a lemma with a proof.
\begin{lemma}
\label{lmm:inv-sim}
Let $f,g : {\reals}_+ \rightarrow {\reals}_+$ be increasing and bijective functions such that
(a) $f(x)\leq g(x)$ for all $x$, 
(b) $f(x)$ is concave up for all $x$ and 
(c) $f(x)\sim g(x)$ as $x\rightarrow\infty$. 
In this case we have $f^{-1}(x)\sim g^{-1}(x)$ as $x\rightarrow\infty$.
\end{lemma}
\begin{proof}
By assumption we then have 
\[
\frac{f(t)}{t} \leq \frac{g(t)}{(f^{-1}\circ g)(t)}\leq \frac{g(t)}{t}
\]

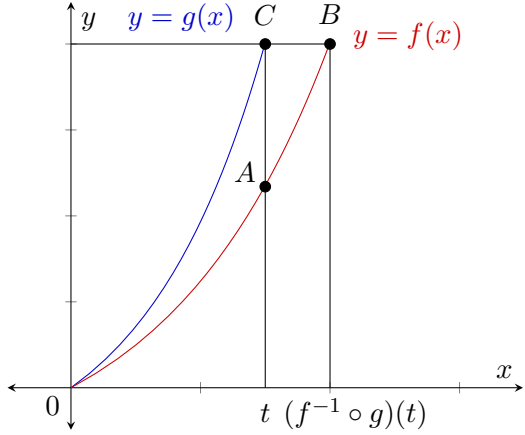
\begin{figure}
\begin{center}
\begin{tikzpicture}[>=stealth]
    \begin{axis}[
        xmin=-.49,xmax=3.49,
        ymin=-.49,ymax=4.49,
        axis x line=middle,
        axis y line=middle,
        axis line style=<->,
        xlabel={$x$},
        ylabel={$y$},
        xticklabels=\empty,
        yticklabels=\empty,
        ]
        \addplot [blue!80!black, domain=1:5] (.93*ln(x), x-1) node[pos=0.75,anchor=south east]{};
        \addplot [blue!80!black, domain=1:5] (1.35, 4.3) node[pos=0.75,anchor=east]{$y=g(x)$};
        \addplot [red!80!black, domain=1:5] (1.24*ln(x), x-1) node[pos=0.35,anchor= west]{};
        \addplot [red!80!black, domain=1:5] (2.1, 4.1) node[pos=0.35,anchor= west]{$y=f(x)$};
        \addplot [black!80!black, domain=0:4] (1.5, x) node[pos=0.15,anchor=east]{};
        \addplot [black!80!black, domain=0:4] (1.5, -.3) node[pos=0.15]{$t$};
        \addplot [black!80!black, domain=0:4] (2, x) node[pos=0.15,anchor=west]{};
        \addplot [black!80!black, domain=0:4] (2.2, -.3) node[pos=0.15]{$(f^{-1} \circ g)(t)$};
        \addplot [black!80!black, domain=0:2] (x, 4) node[pos=0.25,anchor=west]{};
        \addplot [black!80!black, domain=0:2] (1.5, 4) node[pos=0.25,anchor=south, circle]{$C$};
        \addplot [black!80!black, domain=0:2] (2, 4) node[pos=0.25,anchor=south, circle]{$B$};
        \addplot [black!80!black, domain=0:3] (1.5, 2.27) node[pos=0.25,anchor=south east, circle]{$A$};
        \addplot [black!80!black, domain=0:3] (0,0) node[pos=0.25,anchor=north east, circle]{$0$};

        \addplot [mark=*, color=black] table {
        1.5 4
        };
        \addplot [mark=*, color=black] table {
        2 4
        };
        \addplot [mark=*, color=black] table {
        1.5 2.34
        };
    \end{axis}
\end{tikzpicture}
\caption{Slope $\overline{0A} = \frac{f(t)}{t}$ {}$\le$ Slope $\overline{0B}= \frac{g(t)}{(f^{-1} \circ g)(t)}$ $\le$ Slope $\overline{0C}= \frac{g(t)}{t}$}\label{slopegraph}
\end{center}
\end{figure}

for every positive real number $t$ (see Figure~\ref{slopegraph}). Dividing by $g(t)/t$ we obtain
\[
\frac{f(t)}{g(t)} \leq \frac{t}{(f^{-1}\circ g)(t)} \leq 1
\]
for every positive real $t$ and hence $(f^{-1}\circ g)(t)/t$ tends to $1$ as 
$t\rightarrow\infty$. Letting $t = g^{-1}(x)$ and noting that $x\rightarrow\infty$
if and only if $t\rightarrow\infty$ we have $f^{-1}(x)/g^{-1}(x)\rightarrow 1$ as 
$x\rightarrow\infty$.
\end{proof}
The function on the right in (\ref{eqn:f-sim}) has a smooth extension to
the set of all positive real numbers
$g_c(x) = \sqrt{2\pi ec}(x-1/2)^{x-1/2}\left(\frac{c}{e}\right)^{x-1/2}$ 
for every $c\in {\reals}_+$ and it has an inverse $g_c^{-1}(x)$ which we will
express in terms of elementary functions and the Lambert $W$ function.
Recall that over the reals, the curve $z = we^w$ in the Euclidean plane ${\reals}^2$
yields a two-valued inverse $w = W(z)$; the {\em Lambert $W$ function}
whose two branches yield the inverses 
$W_0 : [-1/e,\infty[ \rightarrow [-1,\infty[$ (see Figure~\ref{LambertW}) and
$W_1 : [-1/e,0[ \rightarrow ]-\infty,-1]$. The Lambert $W$ function has been
studied extensively~\cite{Lambert-wiki,Lambert-Wolfram}.

\begin{figure}
\begin{center}
\begin{tikzpicture}[>=stealth]
    \begin{axis}[
        xmin=-2,xmax=19.9,
        ymin=-.9,ymax=2.9,
        axis x line=middle,
        axis y line=middle,
        axis line style=<->,
        xlabel={$x$},
        ylabel={$y$},
        ]
        \addplot [blue!80!black, domain=-1:3] (x * exp(x), x) node[pos=0.25,anchor=south east]{$y=W_0(x)$}; 
    \end{axis}
    
\end{tikzpicture}
\caption{The graph of the upper branch of the Lambert $W$ function $W_0 : [-1/e,\infty[ \rightarrow ]-1,\infty]$.}\label{LambertW}
\end{center}
\end{figure}
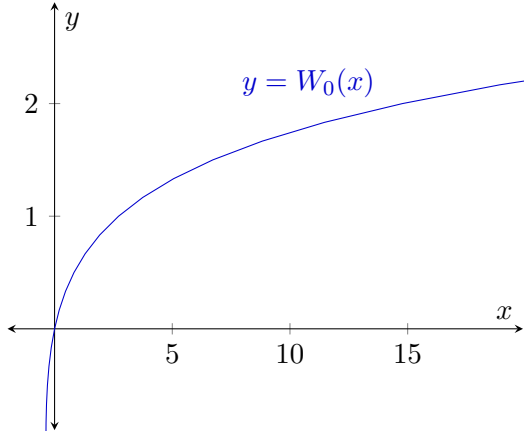

Letting $y = g_c(x)$ and so $x = g_c^{-1}(y)$ we then obtain from (\ref{eqn:f-sim}) 
\[
\frac{y}{\sqrt{2\pi ec}} = \left(\frac{c(x-1/2)}{e}\right)^{x-1/2}
\]
and hence, by putting both sides to the $c/e$-th power we obtain $z = {\omega}^{\omega}$
where 
\[
z:= \left(\frac{y}{\sqrt{2\pi ec}} \right)^{c/e} \mbox{ and } 
\omega := \frac{c(x - 1/2)}{e}.
\]
We then obtain $\log(z) = {\omega}\log(\omega) = \log(\omega)e^{\log(\omega)}$ and hence
$\log(\omega) = W_0(\log(z))$ or 
\[
\log\left(\frac{c(x-1/2)}{e}\right) 
= W_0\left(\frac{c}{e}\log\left(\frac{y}{\sqrt{2\pi ec}}\right)\right).
\]
Solving for $x$ we then have 
\begin{equation}
\label{eqn:W-formula}    
x = g_c^{-1}(y) = 
c^{-1}e^{W_0\left(\frac{c}{e}\log\left(\frac{y}{\sqrt{2\pi ec}}\right)\right) + 1} + 1/2.
\end{equation}
Since $g_c$ restricted to $\nats$ is strictly increasing, this restricted mock-inverse is therefore given by $g_c^*(n) = \lfloor g_c^{-1}(n)\rfloor$ for each $n\in\nats$.
By Claim~\ref{clm:threshold} and Lemma~\ref{lmm:inv-sim} we therefore have the following theorem.

\begin{theorem}
\label{thm:Lambert-main}
\begin{enumerate}
Let $c\in{\reals}_+$ be fixed.
\item The threshold function $f_c(n) = c^n(n-1)!$ has
a mock-inverse $f_c^*(n) \sim g_c^*(n) = \lfloor g_c^{-1}(n)\rfloor$ where $g_c^{-1}(y)$
is given by (\ref{eqn:W-formula}).  
\item Every threshold function $f$ has a mock-inverse $f^*$ that is sandwiched between
two mock-inverses 
$g_{c_1}^*(n) = \lfloor g_{c_1}^{-1}(n)\rfloor$ and 
$g_{c_2}^*(n) = \lfloor g_{c_2}^{-1}(n)\rfloor$ where $g_{c_1}^{-1}(y)$ and
$g_{c_2}^{-1}(y)$ have the form given by (\ref{eqn:W-formula}) for some 
positive real numbers $c_1$ and $c_2$. Further the difference $c_2-c_1$ can be 
arbitrarily small. 
\end{enumerate}
\end{theorem}

In~\cite{Hoorfar-Hassani} it is shown that for every $z\geq e$ we have the following
bounds
\[
\log z - \log(\log z) + \frac{\log(\log z)}{2\log z} 
\leq W_0(z) \leq 
\log z - \log(\log z) + \frac{e}{e-1}\frac{\log(\log z)}{\log z} 
\]
where $\log$ is here the natural logarithm. Consequentially we have
\[
W_0(z) = \log\left(\frac{z}{\log z}\right) + \epsilon(z)\frac{\log(\log z)}{\log z},
\]
where $\epsilon(z)$ is a real function with $1/2 \leq \epsilon(z)\leq e/(e-1)$ for each
$z\geq e$. Therefore we have that
\[
e^{W_0(z)} 
= e^{\log\left(\frac{z}{\log z}\right) + \epsilon(z)\frac{\log(\log z)}{\log z}}
\sim \frac{z}{\log z}
\]
as $z$ tends to infinity. Since we also have 
$\frac{az + b}{\log(az + b)} \sim a\cdot\frac{z}{\log z}$ when $z\rightarrow\infty$ 
for every $a,b\in{\reals}_+$, we then obtain, by the above discussion, that for 
every $c\in{\reals}_+$ when $y$ tends to infinity that 
\begin{eqnarray*}
g_c^{-1}(y) & = & 
  c^{-1}e^{W_0\left(\frac{c}{e}\log\left(\frac{y}{\sqrt{2\pi ec}}\right)\right) + 1} + 1/2 \\
  & \sim & 
  c^{-1}e^{W_0\left(\frac{c}{e}\log\left(\frac{y}{\sqrt{2\pi ec}}\right)\right)} \\
  & \sim & 
  c^{-1}\frac{\frac{c}{e}\log\left(\frac{y}{\sqrt{2\pi ec}}\right)}{\log\left(\frac{c}{e}\log\left(\frac{y}{\sqrt{2\pi ec}}\right)\right)} \\
  & \sim & 
  \frac{1}{e}\frac{\log y}{\log(\log y)}.
\end{eqnarray*}
By Observation~\ref{obs:alpha}, Theorem~\ref{thm:Lambert-main}, the above discussion and 
Observation~\ref{obs:alpha} we then have the following.
\begin{corollary}
\label{cor:Lambert}
Let $\mathbf{T}$ be an LFOR-i-tree with an infinite width and
suppose $\lambda(\mathbf{T}_h)= \lambda_h = f^{*}(h)$ where $f$ is a threshold function. 
In this case we have 
(i) $s(\mathbf{T}_h) = \Theta(h^{\lambda_h})$ for infinitely many $h\in\nats$, and
(ii) the mock-inverse satisfies
\[
f^*(n) \sim \lfloor g_c^{-1}(n)\rfloor \sim 
\left\lfloor\frac{1}{e}\frac{\log n}{\log(\log n)}\right\rfloor,
\]
where $g_c^{-1}(y)$ is given by (\ref{eqn:W-formula}).
\end{corollary}
{\sc Remark:} Note that the last limit function in the above Corollary~\ref{cor:Lambert}
does not depend on $c\in{\reals}_+$ and therefore not on the actual limit of 
$if(i)/f(i+1)$ when $i$ tends to infinity for the threshold function $f$; 
only on the fact that the limit exists as a positive real number.

Our above discussion from Observation~\ref{obs:easy-s-bounds} 
begs the following question that would extend Proposition~\ref{prp:power-h} to an if and only if statement.
\begin{question}
\label{qst:which-function}  
For which functions $h\mapsto f(h)$ do we have for all LFOR-i-trees
$\mathbf{T}$ with $\lambda_h = \Theta(f(h))$
that $s(\mathbf{T}_h)\sim h^{\lambda_h}$ for infinitely
many $h$?
\end{question}

\subsection{Theoretical bounds}
We now turn our attention to the property of an exact theoretical bound.
An noted before, if $\mathbf{T}$ is the LFOR-i-tree with
$\lambda_h = \lambda$; a constant for each $h\geq 1$,
then $s(\mathbf{T}_h) = (h+1)^{\lambda}$. Hence, with this and
Theorem~\ref{thm:poly-iff} in mind, it seems appropriate
to use $h+1$ as the base number for the function $s(\mathbf{T}_h)$
when viewed as a power function when $h$ tends to infinity.
Suppose now an LFOR-i-tree $\mathbf{T}$ has an infinite width
$w(\mathbf{T}) = \lambda(\mathbf{T}) = \infty$.
Let $\ell(h)$ be the unique real function
such that $s(\mathbf{T}_h) = (h+1)^{\ell(h)}$. The main purpose of
this subsection is to prove the following:
\begin{theorem}
\label{thm:lambda-infty}
If $\mathbf{T}$ LFOR-i-tree with an infinite width
$w(\mathbf{T}) = \lambda(\mathbf{T}) = \infty$ and
$s(\mathbf{T}_h) = (h+1)^{\ell(h)}$ for each $h\in\nats$
then
(i) $\ell(h)\leq \lambda_h$ for each $h$,
(ii) $\ell(h)$ is an increasing function of $h$ and
(iii) $\lim_{h\rightarrow\infty}\ell(h) = \infty$.
\end{theorem}
To prove Theorem~\ref{thm:lambda-infty} we first need a key lemma:
\begin{lemma}
\label{lmm:uniform}
If $s(\mathbf{T}_h) \geq (h+1)^k$ then $s(\mathbf{T}_{h+1})\geq (h+2)^k$.   
\end{lemma}
To prove Lemma~\ref{lmm:uniform} we need a technical fact from
elementary function theory.
\begin{claim}
\label{clm:abgamma}
If $\beta > \alpha > 1$ and $\gamma \geq 1$ are real constants,
then the real function $g(x) = (x - \gamma)^{\beta} - x^{\alpha}$
is strictly increasing for $x\geq \gamma$.
\end{claim}
\begin{proof}
Clearly we have
\[
g(x) = x^{\alpha}\left((x-\gamma)^{\beta-\alpha}
\left(1 - \frac{\gamma}{x}\right)^{\alpha} - 1\right).
\]
Since each of the functions $x^{\alpha}$, $(x-\gamma)^{\beta-\alpha}$ and
$\left(1 - \frac{\gamma}{x}\right)^{\alpha}$ are strictly increasing for
$x\geq\gamma$, then so is $g(x)$.
\end{proof}
\begin{proof}[Proof of Lemma~\ref{lmm:uniform}]
We will use induction on $\lambda_h = \lambda(\mathbf{T}_h)$.
If $\lambda_h = 1$, then $s(\mathbf{T}_h) = h+1$ and so
$s(\mathbf{T}_{h+1})\geq h+2$ (in fact we have equality if $\lambda_{h+1} = 1$ as well).

Let $\lambda\geq 2$ be given and assume the Lemma holds for
$\mathbf{T}_h$ if $\lambda_h\leq\lambda$. For the LFOR-i-tree $\mathbf{T}$
there is a largest $\ell\in \{1,\ldots,h\}$ such that $\lambda_{\ell} = 1$
(and so $\lambda_{\ell+1} \geq 2$.) Let $u_{\ell}$ denote the unique vertex
of $\mathbf{T}$ on level $\ell$. In this case the LFOR-i-tree
$\mathbf{T}(u_{\ell})$ rooted at $u_{\ell}$ can be partitioned into two
LFOR-i-trees $\mathbf{T}_1(u_{\ell})$ and $\mathbf{T}_2(u_{\ell})$,
both rooted at $u_{\ell}$ and otherwise vertex disjoint. Note that vertices
on level $h$ in $\mathbf{T}$ will be on level $h-\ell$ in either
$\mathbf{T}_1(u_{\ell})$ or $\mathbf{T}_2(u_{\ell})$. From this we see
that $\lambda(\mathbf{T}_1(u_{\ell})_{h-\ell})
+ \lambda(\mathbf{T}_2(u_{\ell})_{h-\ell}) = \lambda(\mathbf{T}_h) = \lambda_h$
where each $\lambda(\mathbf{T}_i(u_{\ell})_{h-\ell})\geq 1$ for $h\geq \ell+1$. Hence, induction
hypothesis applies to the LFOR-i-tree $\lambda(\mathbf{T}_i(u_{\ell})_h)$
for $i=1,2$. As displayed in (\ref{eqn:s-poly}) (where the role of $\ell$
is here in this proof is shifted downward one step) we then have
\begin{equation}
\label{eqn:ell-induction}
s(\mathbf{T}_h) =
\ell + s(\mathbf{T}_1(u_{\ell})_{h-\ell})_h)s(\mathbf{T}_2(u_{\ell})_{h-\ell}).
\end{equation}
For each $s(\mathbf{T}_i(u_{\ell})_{h-\ell})$ let $\alpha_i\in{\reals}_+$ be such that
$s(\mathbf{T}_i(u_{\ell})_{h-\ell}) = (h-\ell+1)^{\alpha_i}$ and therefore
\[
s(\mathbf{T}_{h}) = 
\ell + s(\mathbf{T}_1(u_{\ell})_{h-\ell})_h)s(\mathbf{T}_2(u_{\ell})_{h-\ell})
= \ell + (h-\ell+1)^{\alpha_1 + \alpha_2}.
\]
Writing $s(\mathbf{T}_{h})
= \ell + (h-\ell+1)^{\alpha_1 + \alpha_2} = (h+1)^{\alpha}$ for
an appropriate $\alpha\in{\reals}_+$, we then get since $h\geq \ell$ by 
(\ref{eqn:ell-induction}) and induction hypothesis that
\[
s(\mathbf{T}_{h+1}) =
\ell + s(\mathbf{T}_1(u_{h-\ell+1})_h)s(\mathbf{T}_2(u_{\ell})_{h-\ell+1})
 \geq \ell + (h-\ell+2)^{\alpha_1 + \alpha_2}.
 \]
Since $\ell + (h-\ell+1)^{\alpha_1 + \alpha_2} = (h+1)^{\alpha}$ we get
by Claim~\ref{clm:abgamma} and above display that
 \[
 s(\mathbf{T}_{h+1})
 = \ell + (h-\ell+2)^{\alpha_1 + \alpha_2} \geq (h+2)^{\alpha}.
 \]
Therefore if $s(\mathbf{T}_h) = (h+1)^{\alpha} \geq (h+1)^k$
then $s(\mathbf{T}_{h+1})\geq (h+2)^{\alpha}\geq (h+2)^k$.
this completes the inductive argument.
\end{proof}
We can now prove Theorem~\ref{thm:lambda-infty}.
\begin{proof}[Proof of Theorem~\ref{thm:lambda-infty}]
By Observation~\ref{obs:easy-s-bounds} we have
$s(\mathbf{T}_h)\leq (h+1)^{\lambda_h}$ and hence we have (i).

By Lemma~\ref{lmm:uniform} we have
$s(\mathbf{T}_{h+1}) = (h+2)^{\ell(h+1)}\geq (h+2)^{\ell(h)}$
and so $\ell(h+1)\geq \ell(h)$ for each $h\geq 1$. Hence
$\ell(h)$ is an increasing function of $h$.

Lastly, since $w(\mathbf{T}) = \lambda(\mathbf{T}) = \infty$,
then by Theorem~\ref{thm:poly-iff} there is an increasing
sequence $(h_k)_{k\geq 1}$ such that
$h_k\rightarrow\infty$ as $k\rightarrow\infty$ and
$s(\mathbf{T}_{h_k}) > (h_k+1)^k$ for each $k\in\nats$.
By definition of $\ell(h)$ we have $\ell(h_k)\geq k$
for each $k\in\nats$. Since $\ell(h)$ is increasing in terms of
$h$ have have $\lim_{h\rightarrow\infty}\ell(h) = \infty$.
\end{proof}

We conclude this section with a few observations. Firstly, we note
that if $\mathbf{T}$ is an LFOR-tree (possibly with some f-vertices)
of finite width $w(\mathbf{T})$, then $\mathbf{T}$ can only
have finitely many leaves since any set of leaves forms an
antichain in $\mathbf{T}$.
It follows that $\mathbf{T}$ has only finitely many f-vertices and hence
there is an $h_1\in\nats$ such that $\mathbf{T}_{h_1}$ contains all the leaves
of $\mathbf{T}$. From the fact that $\mathbf{T}$ has only finitely many f-vertices, we also have the integer sequence $(\lambda(\mathbf{T}_h))_{h \ge 0}$ is eventually increasing. Thus, as defined in (\ref{eqn:lambda}), $\lambda(\mathbf{T})$ is also finite since any set of
vertices on the same level forms an antichain. In particular, there is an $h_2\geq h_1$  such that
$\lambda(\mathbf{T}) = \lambda(\mathbf{T}_h)$ for all $h\geq h_2$.
It follows that each vertex $u\in V_{h_2}(\mathbf{T})$ on level $h_2$
has exactly one parent and one child. In particular, for $h\geq h_2$ each
$\mathbf{T}(u)_{h-h_2}$
is therefore a simple path emanating downward from $u$ of length $h-h_2$
and so $s(\mathbf{T}(u)_{h-h_2}) = h - h_2 + 1$.
On the other hand, for each vertex
$u\in V_k(\mathbf{T})$ on level $k < h_2$ we have by 
Lemma~\ref{lmm:OR-subtrees} that
\begin{equation}
\label{eqn:lower-level}  
s(\mathbf{T}(u)_h) = \prod_{v\in C(u)}(1 + s(\mathbf{T}(v)_{h-1})),
\end{equation}
where each $v\in C(u)$ is on a strictly lower level than $u$ in $\mathbf{T}$,
keeping in mind that $s(\mathbf{T}(u)_h) = 1$ when $u$ is a leaf. 
By starting on level zero at the root $u = r$ with
$s(\mathbf{T}_h) = s(\mathbf{T}(r)_h)$ and replacing each instance
of $s(\mathbf{T}(u)_h)$ using (\ref{eqn:lower-level}), we
obtain a finite algebraic expression of $s(\mathbf{T}_h)$ in terms of the
$\lambda(\mathbf{T})$ expressions $s(\mathbf{T}(u)_{h-h_2})$ for
each vertex $u\in V_{h_2}(\mathbf{T})$ on level $h_2$.
In this way we obtain similarly to Observation~\ref{obs:s-poly} and
Theorem~\ref{thm:poly-iff} the following.

\begin{proposition}
\label{prp:s-poly-LFOR}
If $\mathbf{T}$ is an LFOR-tree with a finite width $w(\mathbf{T})\in\nats$,
then there is an $N\in\nats$ such that for all $h\geq N$
the number $s(\mathbf{T}_h)$ of OR-subtrees of height $h$ or less
is a polynomial in $h$ of degree $\lambda(\mathbf{T})$.
\end{proposition}

Unlike LFOR-i-trees the converse does not hold: there is an
LFOR-tree $\mathbf{T}$ with an infinite width $w(\mathbf{T})$
such that $s(\mathbf{T}_h)$ is bounded above by a fixed polynomial
in $h$.

\begin{center}
\begin{figure}

\begin{tikzpicture}[scale=.75]
 
\node [style={draw=black,circle,fill}, scale=.25, label=$u_0$] {}
     child {[fill] circle (2pt)
    }
    child { node [fill, circle, scale=.25, label=right:$u_1$]{} circle (2pt)
      child {[fill] circle (2pt)} 
      child {node [fill, circle, scale=.25, label=right:$u_2$]{} circle (2pt)
      child[missing]
      child {node [fill, circle, scale=.25, label=right:$u_3$]{} circle (2pt)
      child {[fill] circle (2pt)}
      child {node [fill, circle, scale=.25, label=right:$u_4$]{} circle (2pt)
      child[missing]
      child {node [fill, circle, scale=.25, label=right:$u_5$]{} circle (2pt) 
      child[missing]
      child {node [fill, circle, scale=.25, label=right:$u_6$]{} circle (2pt) 
      child[missing]
      child {node [fill, circle, scale=.25, label=right:$u_7$]{} circle (2pt) 
      child {[fill] circle (2pt)} 
      child {node [fill, circle, scale=.25, label=right:$u_8$]{} circle (2pt) 
      child[missing]
      child {node [fill, circle, scale=.25, label=right:$u_9$]{} circle (2pt) {node[below, rotate=25]{$\vdots$}}}}}}}}}}
    };
 
\end{tikzpicture}
\caption{ \label{counterex}}
\end{figure}
\end{center}

\begin{example}
\label{exa:log-leaves}  
Consider an LFOR-tree $\mathbf{T}$ formed by a downward simple
path from the root at level zero $r = u_0,u_1,\ldots$ where we add one leaf child on each level $2^i$ for $i\geq 0$ to its parent
$u_{2^i-1}$ (see Firgure~\ref{counterex}). For each $h\geq 0$ there are $2^{\lfloor\lg(h+1)\rfloor + 1}$
OR-subtrees of $\mathbf{T}$ that contain $u_h$ and not $u_{h+1}$. 
It follows that 
\[
s(\mathbf{T}_h)
\le \sum_{i=0}^h2^{\lfloor\lg(i+1)\rfloor+1}
< (h+1)\cdot 2^{\lfloor\lg(h+1)\rfloor+1}
\le
(h+1)\cdot 2^{\lg(h+1)+1}
= 2(h+1)^2,
\]
a polynomial in $h$ of degree $2$. (In fact, the first inequality is equality for $h \neq 2^i-1$.) 
Yet $\mathbf{T}$ has an infinite
antichain containing all the leaves on levels $2^i$ where $i\geq 0$.
Therefore $\mathbf{T}$ is an LFOR-tree with an infinite width
$w(\mathbf{T})$ where $s(\mathbf{T}_h)$ is bounded above by a fixed
polynomial $2(h+1)^2$.
\end{example}
In the following section we will discuss LFOR-trees more
systematically and both sharpen and generalize some results from
this section.

\section{General LFOR-trees}
\label{sec:LFOR-trees}

Every LFOR-tree $\mathbf{T}$ has a bipartition
of its vertices into the i-vertices and f-vertices. The root of
$\mathbf{T}$ is always an i-vertex since every vertex is a descendant
of the root and $\mathbf{T}$ is an infinite tree. Every i-vertex has an
i-vertex child and every child of an f-vertex is an f-vertex. Each parent
of an i-vertex is clearly also an i-vertex and therefore any two i-vertices
have a common ancestor that is also an i-vertex. This implies that
the subtree of $\mathbf{T}$ induced by all the i-vertices is connected. If the parent of an f-vertex $u$ is an i-vertex then all the descendants of
the $u$ are also f-vertices and $u$ generates a maximal finite
subtree $\mathbf{T}(u)$ all of which vertices are f-vertices.
\begin{definition}
\label{def:i-tree-f-trees}  
Let $\mathbf{T}$ be an LFOR-tree.

(i) The {\em i-subtree} $\mathbf{T}^i$ of $\mathbf{T}$ is the
unique subtree induced by all the i-vertices of $\mathbf{T}$. 

(ii) A subtree of $\mathbf{T}$ all of which vertices are f-vertices
is an {\em f-subtree}.
\end{definition}
The following observations easily verified.
\begin{observation}
\label{obs:i-f-subtrees}
Let $\mathbf{T}$ be an LFOR-tree.  

(i) The i-subtree $\mathbf{T}^i$ of an LFOR-tree $\mathbf{T}$
is and LFOR-i-tree with the same root as $\mathbf{T}$.

(ii) Every f-subtree of $\mathbf{T}$ is finite and is a subtree of
a unique maximal f-subtree induced by a maximal f-vertex in the natural
partial ordering of $\mathbf{T}$ and its descendants.

(iii) There is a unique countable disjoint collection of maximal
f-subtrees of $\mathbf{T}$ rooted at the maximal f-vertices of $\mathbf{T}$.
These roots are the f-vertices with an i-vertex as a parent.
\end{observation}
Denoting the set of all maximal f-vertices of $\mathbf{T}$ by
$F(\mathbf{T})$ we then have by Observation~\ref{obs:i-f-subtrees} 
a unique partition of $\mathbf{T}$ into its i-subtree and maximal f-subtrees
\begin{equation}
\label{eqn:LFOR-part}  
\mathbf{T} = \mathbf{T}^i\cup\bigcup_{u\in F(\mathbf{T})}\mathbf{T}(u).
\end{equation}
In the above partition all the i-vertices of $\mathbf{T}$ are in
$\mathbf{T}^i$ and each f-vertex is contained in exactly one maximal
f-subtree $\mathbf{T}(u)$ for some maximal f-vertex $u\in F(\mathbf{T})$. By
(\ref{eqn:LFOR-part}), Observation~\ref{obs:w=l} and
Corollary~\ref{cor:cup-w} we therefore obtain a minimal
decomposition of $\mathbf{T}$ into infinite $C_*^{\infty}$
and finite $C_*^f$  chains as
\begin{equation}
\label{eqn:LFOR-chains}  
\mathbf{T} =
\left(\bigcup_{\alpha=1}^{w(\mathbf{T}^i)}C^{\infty}_{\alpha}\right)
\cup
\left(\bigcup_{u\in F(\mathbf{T})}\bigcup_{\beta=1}^{l_u}C^f_{u;\beta}\right),
\end{equation}
where $l_u = |L(\mathbf{T}(u))|$ denote the number of
leaves of the f-subtree $\mathbf{T}(u)$. The Dilworth-like
decomposition of $\mathbf{T}$
in (\ref{eqn:LFOR-chains}) is minimal in the sense that there is no such
decomposition with fewer than $w(\mathbf{T}^i)$ infinite chains or
fewer than $l(\mathbf{T}) = \sum_{u\in F(\mathbf{u})}l_u$ finite chains.
If either $l(\mathbf{T})$ or $w(\mathbf{T}^i)$ is infinite, then by
Corollary~\ref{cor:cup-w} $\mathbf{T}$ contains an infinite antichain.
If both $l(\mathbf{T})$ and $w(\mathbf{T}^i)$ are finite, so $\mathbf{T}$
has only finitely many leaves, then there is an $h_1\in\nats$ such that
all the leaves $L(\mathbf{T})$ of $\mathbf{T}$ are on levels
strictly less than $h_1$.
Also, since $w(\mathbf{T}^i) = \lambda(\mathbf{T}^i)$ as in
(\ref{eqn:lambda}) there is an $h_2 > h_1$ such that
$\lambda(\mathbf{T}^i) = \lambda(\mathbf{T}^i_{h_2})$ in which case
$V_{h_2}(\mathbf{T}^i)$ is an antichain in $\mathbf{T}$. By our choice of
$h_2$ no leaves are descendants of vertices of $V_{h_2}(\mathbf{T}^i)$ and
so the disjoint union $V_{h_2}(\mathbf{T}^i)\cup L(\mathbf{T})$ forms
an antichain in $\mathbf{T}$. In every case we have the following.
\begin{claim}
\label{clm:LFOR-w}  
For an LFOR-tree $\mathbf{T}$ we have that
$w(\mathbf{T}) = w(\mathbf{T}^i) + l(\mathbf{T})$.
\end{claim}
\begin{example}
\label{exa:2-caterp}
Consider a version of a previous Example~\ref{exa:k-caterp} where $k=2$ and
let $\mathbf{T}$ be an LFOR-tree formed by the infinite path emanating
from the root downward and where we have added one leaf on each level
$\geq 1$. Since $\mathbf{T}$ has infinitely many leaves
$l(\mathbf{T}) = \infty$, its width $w(\mathbf{T}) = \infty$ as well.
Hence, since $w(\mathbf{T}^i) = 1$, we have
$w(\mathbf{T}) = w(\mathbf{T}^i) + l(\mathbf{T})$. However, no
infinite antichain of $\mathbf{T}$ can contain a vertex from
$\mathbf{T}^i$ which is the unique infinite path in $\mathbf{T}$
emanating downward from the root, since any vertex on this path
is incomparable to only finitely many leaves in $\mathbf{T}$.
Consequentially, an antichain of cardinality $w(\mathbf{T})$
does not need to contain any vertices from $\mathbf{T}^i$.
\end{example}
\vspace{3 mm}

{\sc Remark:} Claim~\ref{clm:LFOR-w} is for a general LFOR-tree $\mathbf{T}$
solely a cardinality result.
Only when both $w(\mathbf{T}^i)$ and $l(\mathbf{T})$
are finite can we obtain an explicit antichain in $\mathbf{T}$ that
contains $w(\mathbf{T}^i)$ vertices from the i-subtree $\mathbf{T}^i$
and the $l(\mathbf{T})$ leaves of $\mathbf{T}$.
\begin{definition}
\label{def:fat-thin}
Let $\mathbf{T}$ be an LFOR-tree.

(i) We say that $\mathbf{T}$ is {\em thin} if
$w(\mathbf{T}^i) = \lambda(\mathbf{T}^i)\in\nats$ is finite.

(ii) We say that $\mathbf{T}$ is {\em fat} if
$w(\mathbf{T}^i) = \lambda(\mathbf{T}^i) = \infty$ is infinite.
\end{definition}

\begin{figure}
\begin{center}
\begin{tabular}{ccc}

\begin{tikzpicture}

\node[draw=none, fill=none, label=below:$\vdots$] at (-1, -6)   (a) {};
\node[draw=none, fill=none, label=below:$\vdots$] at (1, -6)   (b) {};

\draw[thick] (-1,-3) -- (0,0) -- (1,-3)
node[pos=0,style={draw=black,circle,fill}, scale=.25,fill,label=above:$r$]{};
\draw[thick] (-1,-3) -- (-1,-6);
\draw[thick] (1,-3) -- (1,-6);
\end{tikzpicture} & \hspace{50pt} & 
\begin{tikzpicture}
\node[draw=none, fill=none, label=below left:\rotatebox{160}{$\vdots$}] at (-1.78, -6)   (a) {};
\node[draw=none, fill=none, label=below right:\rotatebox{198}{$\vdots$}] at (1.78, -6)   (b) {};

\draw[thick] (-2,-6) -- (0,0) -- (2,-6)
node[pos=0,style={draw=black,circle,fill}, scale=.25,fill,label=above:$r$]{};

\end{tikzpicture} 

\end{tabular}

\caption{A thin LFOR-tree on the left and a fat LFOR-tree on the right.}\label{thinfattrees}

\end{center}
\end{figure}
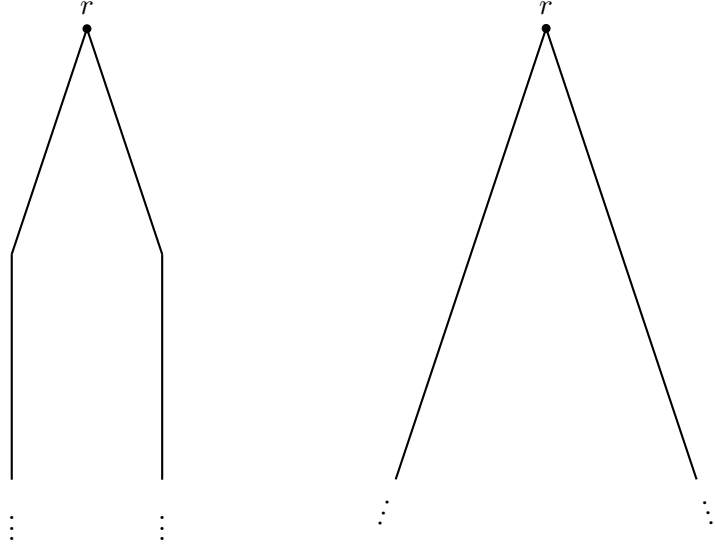

Examples of thin and fat LFOR-trees are pictured in Figure~\ref{thinfattrees}. Since $\mathbf{T}^i\subseteq \mathbf{T}$ is a subtree we have for
any $h\in\nats$ that $s(\mathbf{T}_h) \geq s(\mathbf{T}^i_h)$ and
hence by Theorem~\ref{thm:poly-iff} the following.
\begin{proposition}
\label{prp:fat-not-poly}  
If $\mathbf{T}$ is a fat LFOR-tree, then $s(\mathbf{T}_h)$ is not bounded
from above by any fixed polynomial in $h$.
\end{proposition}
\begin{example}
\label{exa:2-caterp-exp}
Consider the LFOR-tree $\mathbf{T}$ from previous
Example~\ref{exa:2-caterp} formed by the infinite path emanating
from the root downward and where
we have added one leaf on each level $\geq 1$. In this case
$\mathbf{T}$ is a thin tree with $w(\mathbf{T}^i) = 1$ and it has
infinitely many leaves. By Lemma~\ref{lmm:OR-subtrees}
$s(\mathbf{T}_h)$ satisfies the recursion $s(\mathbf{T}_0) = 1$ and
$s(\mathbf{T}_h) = 2(1 + s(\mathbf{T}_{h-1}))$ and so
$s(\mathbf{T}_h) = 3\cdot 2^h - 2$ which is exponential in $h$.
So, clearly, thinness is not sufficient to guarantee for $s(\mathbf{T}_h)$
to grow polynomially in $h$.
\end{example}

Our next objective is to prove the following theorem.
\begin{theorem}
\label{thm:LFOR-thin-poly}
Let $\mathbf{T}$ be an LFOR-tree with a finite width $w(\mathbf{T})\in\nats$
as a poset. Then $\mathbf{T}$ is thin and has finitely many maximal f-vertices,
say $u_1,\ldots,u_f$. Moreover $s(\mathbf{T}_h)$ is a polynomial
in $h$ of degree $\lambda(\mathbf{T}) = w(\mathbf{T}^i)$
and with a leading coefficient of $\prod_{j=1}^f(1 + s(\mathbf{T}(u_j)))$
for all sufficiently large $h\in\nats$. 
\end{theorem}
\begin{proof}
Let $\mathbf{T}$ be an LFOR-tree with a finite width $w(\mathbf{T})$.  
By Claim~\ref{clm:LFOR-w} $w(\mathbf{T}^i)$ is finite, $\mathbf{T}$
has finitely many leaves, and hence
$\lambda(\mathbf{T}) = \lambda(\mathbf{T}^i) = w(\mathbf{T}^i)$.
As in the discussion right before
Claim~\ref{clm:LFOR-w} there are $h_1$ and $h_2$ such that 
(i) all the leaves of $\mathbf{T}$ are on levels strictly less
than $h_1$, and
(ii) $h_2 > h_1$ where $\lambda(\mathbf{T}^i) = \lambda(\mathbf{T}^i_{h_2})$.
If $N > h_1,h_2$ is fixed and $h\geq N$, then we can partition the
OR-subtrees $S(\mathbf{T}_h)$ of $\mathbf{T}_h$ into
$2^{\lambda(\mathbf{T})}$ classes, one class $S_B(\mathbf{T}_h)$ for each subset
$B\subseteq V_N(\mathbf{T})$ of vertices on level $N$ in $\mathbf{T}$
containing the OR-subtrees $T\subseteq \mathbf{T}_h$ with
$V_N(T)\cap V_N(\mathbf{T}) = B$. As this partition is disjoint we have
\begin{equation}
\label{eqn:set-part}  
s(\mathbf{T}_h) = \sum_{B\subseteq V_N(\mathbf{T})}s_B(\mathbf{T}_h),
\end{equation}
where $s_B(\mathbf{T}_h) = |S_B(\mathbf{T}_h)|$ for each subset $B$.
We note a few things:

(i) For each $B\subseteq V_N(\mathbf{T})$ and each $h\geq N$ each
OR-subtree $T\in S_B(\mathbf{T}_h)$ is obtained from an OR-subtree
$T'\in S_B(\mathbf{T}_N)$ by appending to $T'$ paths of length at most $h-N$
downward from each vertex in $B$. Therefore we have
$s_B(\mathbf{T}_h) = s_B(\mathbf{T}_N)(h - N + 1)^{|B|}$. By
(\ref{eqn:set-part}) we therefore have
\[
s(\mathbf{T}_h)
= \sum_{B\subseteq V_N(\mathbf{T})}s_B(\mathbf{T}_N)(h - N + 1)^{|B|},
\]
a polynomial in $h$ of degree $\lambda(\mathbf{T})$.
Since $b = |B| \in \{0,1,\ldots,\lambda(\mathbf{T})\}$, we can further
write the above sum as the following
\begin{equation}
  \label{eqn:deg-part}
s(\mathbf{T}_h)
= \sum_{b=0}^{\lambda(\mathbf{T})}\left(
\sum_{\substack{B\subseteq V_N(\mathbf{T}) \\ |B| = b}}
s_B(\mathbf{T}_N)\right)(h - N + 1)^b.
\end{equation}

(ii) The sum for $B = \emptyset$ in (\ref{eqn:deg-part}) 
is just one term
$s_{\emptyset}(\mathbf{T}_h) = s(\mathbf{T}_N)$ for every $h\geq N$,
and hence it contributes only to the constant term of $s(\mathbf{T}_h)$
as a polynomial in $h$.

(iii) Likewise the sum for $B = V_N(\mathbf{T})$ in (\ref{eqn:deg-part})
is just one term
\[
s_{V_N(\mathbf{T})}(\mathbf{T}_h)
= s_{V_N(\mathbf{T})}(\mathbf{T}_N)(h - N + 1)^{\lambda(\mathbf{T})}
\]
that determines the leading coefficient of $s(\mathbf{T}_h)$.
We note that each OR-subtree in $S_{V_N(\mathbf{T})}(\mathbf{T}_N)$
contains all the i-vertices of $\mathbf{T}$ on levels $0,1,\ldots,N$
and therefore the number $s_{V_N(\mathbf{T})}(\mathbf{T}_N)$ is then
completely determined by all the OR-f-subtrees of $\mathbf{T}_N$
rooted at the maximal f-vertices $u_1,\ldots,u_f$ w.r.t.~the
natural partial order of $\mathbf{T}$. Since each $\mathbf{T}(u_j)$
has $s(\mathbf{T}(u_j))$ OR-subtrees (each rooted at $u_i$) there
are $1 + s(\mathbf{T}(u_j))$ OR-subtrees when we include the empty
OR-subtree. Therefore
$s_{V_N(\mathbf{T})}(\mathbf{T}_N) = \prod_{j=1}^f(1 + s(\mathbf{T}(u_j)))$
and so (\ref{eqn:deg-part}) becomes
\[
s(\mathbf{T}_h) =
\prod_{j=1}^f(1 + s(\mathbf{T}(u_j)))(h - N + 1)^{\lambda(\mathbf{T})} + 
\sum_{b=1}^{\lambda(\mathbf{T})-1}\left(
\sum_{\substack{B\subseteq V_N(\mathbf{T}) \\ |B| = b}}
s_B(\mathbf{T}_N)\right)(h - N + 1)^b
+ s(\mathbf{T}_N),
\]
which completes the proof.
\end{proof}  
As with LFOR-i-trees of finite width there is an $N\in\nats$ (here any
$N\geq h_2$ from the above proof will do) such that all leaves are on
levels strictly less than $N$ and
$w = w(\mathbf{T}^i) = \lambda(\mathbf{T}^i) = \lambda(\mathbf{T}^i_h)$
for all $h\geq N$. Let $M = |V(\mathbf{T}_N)|$ be the number of vertices of
$\mathbf{T}_N$. In this case, $\mathbf{T}_h$ has exactly $n_h = M + w(h-N)$
vertices for each $h\geq N$. As the height $h$ can be written in terms of
the number $n_h$ of vertices of $\mathbf{T}_h$ we have
by the above Theorem~\ref{thm:LFOR-thin-poly} the following.
\begin{corollary}
\label{cor:poly-vert}
Let $\mathbf{T}$ be an LFOR-tree of finite width as a poset and
therefore with finitely many maximal f-vertices, say $u_1,\ldots,u_f$.
In this case the number $s(\mathbf{T}_h)$ of OR-subtrees of $\mathbf{T}$  
is a polynomial in terms of $n_h = |V(\mathbf{T}_h)|$, the number of
vertices of $\mathbf{T}_h$, for all sufficiently large enough $h$
of degree $w = w(\mathbf{T}^i)$ with leading coefficient
of $\frac{1}{w^w}\prod_{j=1}^f(1 + s(\mathbf{T}(u_j)))$.
\end{corollary}

\section{Summary}
\label{sec:summary}

We briefly discuss some of the main results in this article. The novel
contribution in this paper are from
Sections~\ref{sec:or-subtrees}, \ref{sec:infinite-OR-trees}
and~\ref{sec:LFOR-trees}.

In Section~\ref{sec:or-subtrees} we mimicked and slightly generalized methods
used in~\cite{Aho-Sloane} to obtain an exact formula for the number
of ordered rooted subtrees of a complete $k$-ary ordered rooted tree of
a given height $h\in\nats$. This formula, given in
Corollary~\ref{cor:t-closed-form} has a base number $c(k)\in\reals$ which
we then analyze and obtain a sharp asymptotic expression for it given
in Theorem~\ref{thm:c(k)}. This implies, in particular that $c(k)$ tends to
$2$ from above when $k$ tends to infinity.

In Section~\ref{sec:infinite-OR-trees} we defined locally finite ordered
rooted trees (LFOR-trees) and then LFOR-i-trees as LFOR-trees where
each vertex has infinitely many descendants. The first main result in this
section is Theorem~\ref{thm:poly-iff} which describes completely when
an LFOR-i-tree $\mathbf{T}$
has $s(\mathbf{T}_h)$; the total number of OR-subtrees of $\mathbf{T}_h$,
bounded by a polynomial in $h$. This occurs
iff $s(\mathbf{T}_h)$ is itself a polynomial in $h$,
iff $\mathbf{T}$ has a finite width $w(\mathbf{T})\in\nats$
when viewed as an infinite poset in a natural way
iff $\mathbf{T}$ can be decomposed into $w(\mathbf{T})$ infinite chains.
We then derive a simply expressible bounds for $s(\mathbf{T}_h)$ given
in Observation~\ref{obs:easy-s-bounds} and then obtain the threshold function
in terms of $h$ for when exactly the mentioned easy bounds will yield
a power function $h\mapsto h^{\lambda_h}$ that is asymptotically tight
for infinitely many integer values $h$ in Proposition~\ref{prp:power-h}.
Here $\lambda_h = \lambda(\mathbf{T}_h)$ is a parameter that
is easily obtained directly from $\mathbf{T}$. 
We express these threshold functions $h\mapsto\lambda_h$ in terms 
of the Lambert W~function and note that in Corollary~\ref{cor:Lambert}
they are all equivalent in that the ratio between any two of them tends
to $1$ when $h$ tends to infinity. 
This section was then concluded with
Theorem~\ref{thm:lambda-infty} which describes the properties of the function
$\ell(h)$ for which we have $s(\mathbf{T}_h) = (h+1)^{\ell(h)}$. 

In Section~\ref{sec:LFOR-trees} we considered general LFOR-trees and noted
that every LFOR-tree $\mathbf{T}$ has a unique (and maximum) LFOR-i-subtree
$\mathbf{T}^i$ consisting of all the vertices of $\mathbf{T}$ that have
infinitely many descendants. We showed that if $\mathbf{T}$ is fat, then
$s(\mathbf{T}_h)$ is never bounded from above by a polynomial in $h$. However,
not all thin LFOR-trees have $s(\mathbf{T}_h)$ bounded by a polynomial in $h$.
We showed that when the width $w(\mathbf{T})$ is finite, then $\mathbf{T}$
is thin and $s(\mathbf{T}_h)$ is a polynomial of degree $w(\mathbf{T}^i)$ with
a leading coefficient that can easily be expressed in terms of $\mathbf{T}$
for sufficiently large $h$.

\bibliographystyle{amsalpha}
\bibliography{Geirbib}

\end{document}